\documentclass[final]{siamltex}

\usepackage[ngerman,english]{babel} 
\usepackage{color}

\usepackage{amssymb,amsmath,amsfonts} 
\usepackage{mathrsfs}
\usepackage{nicefrac}
\usepackage{graphicx}
\usepackage{caption}
\usepackage{subcaption}

\newtheorem{ex}[theorem]{Example}

\newtheorem{remark}[theorem]{Remark}

\newcommand{\TimesFont}{\RequirePackage{times}\RequirePackage[scaled=0.92]{helvet}}
\TimesFont

\newcommand{\cA}{\mathcal{A}}

\newcommand{\cC}{\mathcal{C}}
\newcommand{\cD}{\mathcal{D}}

\newcommand{\cF}{\mathcal{F}}

\newcommand{\cH}{\mathcal{H}}

\newcommand{\cN}{\mathcal{N}}
\newcommand{\cO}{\mathcal{O}}

\newcommand{\cR}{\mathcal{R}}

\newcommand{\cU}{\mathcal{U}}
\newcommand{\cV}{\mathcal{V}}

\newcommand{\E}{\mathbb{E}}

\newcommand{\N}{\mathbb{N}}

\newcommand{\R}{\mathbb{R}}

\usepackage{color}

\definecolor{darkred}{RGB}{139,0,0}
\definecolor{darkgreen}{RGB}{0,100,0}
\definecolor{darkmagenta}{RGB}{139,0,139}
\definecolor{darkpurple}{RGB}{110,0,180}
\definecolor{darkblue}{RGB}{40,0,200}
\definecolor{darkorange}{RGB}{255,140,0}
\begin{document}\selectlanguage{english}

\title{Fokker--Planck particle systems for Bayesian inference: Computational approaches}
\author{Sebastian Reich\thanks{Universit\"at Potsdam, 
Institut f\"ur Mathematik, Karl-Liebknecht-Str. 24/25, D-14476 Potsdam, Germany ({\tt sebastian.reich@uni-potsdam.de})}  
\and Simon Weissmann\thanks{Universit\"at Heidelberg,
Interdisziplin\"ares Zentrum f\"ur Wissenschaftliches Rechnen, D-69120 Heidelberg, Germany ({\tt simon.weissmann@uni-heidelberg.de})}}

\maketitle

\begin{abstract} 
Bayesian inference  can be embedded into an appropriately defined dynamics in the space of probability measures.
In this paper, we take Brownian motion and its associated Fokker--Planck equation as a starting point for such embeddings and
explore several interacting particle approximations. More specifically, we consider both deterministic and stochastic interacting particle 
systems and combine them with the idea of preconditioning by the empirical covariance matrix. 
In addition to leading to affine invariant formulations which asymptotically speed up convergence, preconditioning 
allows for gradient-free implementations in the spirit of the ensemble Kalman filter. While such gradient-free 
implementations have been demonstrated to work well for posterior measures that are nearly Gaussian, we extend their scope
of applicability to multimodal measures by introducing localised gradient-free approximations.  
Numerical results demonstrate the effectiveness of the considered methodologies.
\end{abstract}

\noindent {\footnotesize
\\
	\noindent {\bf AMS(MOS) subject classifications.} 65N21, 62F15, 65N75, 65C30, 90C56
}

\section{Introduction}

In this paper, we address the problem of how to convert samples from a prior distribution into samples from the posterior
distribution \cite{Kaipio:1338003,stuart_2010} using appropriately defined evolution equations, that is, 
interacting particle systems either in deterministic or stochastic form. Generally speaking, the desired evolution equations can be formulated
such that the transformation from prior to posterior is achieved either over a fixed time interval or, alternatively, the posterior measure 
is approached asymptotically as time goes to infinity. The former approach is linked to coupling of measures ideas and has, for example, 
been explored from a computational perspective in \cite{sr:daum11,sr:reich10,sw:M16}. The latter approach is typically based on stochastic differential equations
and their associated Fokker--Planck equations with the property that the desired posterior measure is invariant under the given
stochastic process \cite{sr:P14}. More specifically, Brownian dynamics leads to a gradient flow in the space of probability measures \cite{sr:JKO98}, which we denote here as Wasserstein dynamics. It is well-known that Wasserstein dynamics minimises the Kullback--Leibler divergence
between the temporal particle distribution and the desired invariant measure. The gradient flow structure of Wasserstein dynamics has been 
explored numerically, for example, in \cite{sr:CCP19,SR2018,SPSR2019}. The key idea is to define the interacting particle systems such 
that a discrete gradient flow structure is maintained. We note that the idea of minimising the Kullback--Leibler divergence arises also in sampling methods via measure transport maps \cite{sw:M16}, where the aim is to construct a deterministic coupling between a probability measure of interest and a simpler reference measure.

Once a gradient flow structure  has been identified (either at the continuum or discretised level), 
one can think about modifications that alter the gradient dynamics without changing the invariant measure. 
In the classical problem of minimising a cost function $V(x)$, this idea corresponds to considering the preconditioned gradient descent dynamics
\begin{equation} \label{eq:GD}
\frac{{\rm d} x}{{\rm d}t} = -B \nabla_x V(x),
\end{equation}
where $B$ is a symmetric positive-definite matrix. In the context of the interacting particle systems considered in this paper, the 
empirical covariance matrix $P_t^{xx}$ of the ensemble of particles emerges as a natural analog to 
the preconditioning matrix $B$ in (\ref{eq:GD}) \cite{AGFHWLAS2019,SPSR2019}. 
See also \cite{sr:GC11,sr:LMW18} for related ideas in the context of Markov chain Monte Carlo
methods and \cite{NKAS2018} in the context of minimisation. 
The preconditioning by $P_t^{xx}$ gives rise to a modified gradient flow structure in the limit of infinitely many particles 
which has been investigated in \cite{AGFHWLAS2019} under the notion of Kalman--Wasserstein gradient flows. The numerical implementation of
Kalman--Wasserstein dynamics have been explored in \cite{AGFHWLAS2019,sr:NR19} 
using appropriate modifications to standard Brownian dynamics. However, as already indicated in \cite{SPSR2019}, the
same preconditioning can be applied to the deterministic interacting particle systems proposed in \cite{sr:CCP19,SR2018} 
and a first contribution of this paper is to put the resulting deterministic and stochastic interacting particle systems within a common 
mathematical framework. We also demonstrate that preconditioning by $P_t^{xx}$ leads to an affine invariant particle dynamics; a property
which has been shown to be important for Markov chain Monte Carlo methods in \cite{sr:GW10}.

Combining preconditioning by $P_t^{xx}$ with gradient-free formulations of the ensemble Kalman filter \cite{sr:evensen,sr:stuart15,sr:reichcotter15}, 
one can derive gradient-free implementations of the interacting particle systems arising from Kalman--Wasserstein dynamics \cite{AGFHWLAS2019}. 
These formulations are no longer of gradient flow structure but have nevertheless been demonstrated to be computationally robust in 
\cite{AGFHWLAS2019} for posterior distributions which are unimodal and close to Gaussian. 
In this paper, we push this approach further by considering localised approximations $P^{xx}_t(X_t^{(i)})$ which only take into account particles in
the vicinity of a given particle $X_t^{(i)}$ \cite{sr:LMW18}. These localised covariance matrices in turn allow for a local approximation of derivatives 
and hence facilitate the application of the resulting gradient-free formulations to multimodal distributions. We demonstrate the effectiveness of these
approximations by means a numerical example with a multimodal posterior distribution. 

In summary, the key contributions of this paper are to (i) provide a unifying framework for preconditioned interacting particle systems
approximating gradient flow structures in the space of probability measures with application to Bayesian inference problems (BIPs) 
and to (ii) introduce localised covariance matrices for Kalman--Wasserstein dynamics that achieve an efficient preconditioning for multimodal 
distributions also facilitating gradient-free implementations.

The remainder of this paper is organised as follows. The considered class of BIPs is specified mathematically
in Section \ref{sec:prob_form} together with a general framework (\ref{eq:gradient_methods}) 
for formulating deterministic as well as stochastic interacting particle
methods. We also summarise the required background on Brownian dynamics and its associated Fokker--Planck equation. The actual
computational implementations of deterministic as well as stochastic interacting particle systems within (\ref{eq:gradient_methods}) 
are presented in Section \ref{sec:computational} and, more precisely, in Sections \ref{sec:FP} (deterministic) 
and \ref{sec:LD} (stochastic), respectively. The generalisation of
the deterministic Wasserstein formulation (\ref{eq:Liouville})--(\ref{eq:RVF}) 
to  Kalman--Wasserstein flows is introduced in Section \ref{sec:KW} together with the novel localised 
preconditioning approach. We also demonstrate the important property of affine invariance \cite{sr:GW10} of the resulting preconditioned
gradient dynamics. The necessary modifications to the associated stochastic interacting particle system formulations with multiplicative noise, 
as first discussed in \cite{sr:NR19}, can be found in Section \ref{sec:LD}. Numerical results are
presented in Section \ref{sec:numerics}. The paper concludes with a summary in Section \ref{sec:conclusions}.


\section{Mathematical problem formulation} \label{sec:prob_form}

We consider the inverse problem of recovering a random variable $X\in\R^{N_x}$ from observations $y\in\R^{N_y}$ which obey
the following forward model:
\begin{equation}\label{eq:IP}
Y = h(X)+\Xi,
\end{equation}
where $h:\R^{N_x}\to\R^{N_y}$ denotes some nonlinear forward map and the mean zero $N_y$-valued Gaussian random variable $\Xi$
represents measurement errors with positive definite error covariance matrix $R \in \mathbb{R}^{N_y\times N_y}$. 

We employ the Bayesian approach to inverse problems and view $(X,Y)$ 
as a jointly varying random variable on $\R^{N_x} \times \R^{N_y}$ with the marginal in $X$ given by the prior distribution $\pi_0$. We assume 
that $\Xi$  and $X$ are independent. Then, by Bayes' Theorem, the solution to the BIP is a $\R^{N_x}$-valued random variable 
$X\mid y\sim \pi(\cdot\, |y)$, where 
\begin{equation} \label{eq:posterior}
\pi({\rm d}x|y)=\frac{1}{C}\exp(-\Phi(x ; y))\,\pi_0({\rm d}x)
\end{equation}
with $C>0$ a normalisation constant such that $\pi({\rm d}x|y)$ is a probability measure, that is,
\begin{equation}
C:=\int_{\R^{N_x}}\exp(-\Phi(x;y))\,\pi_0({\rm d}x),
\end{equation}
and with $\Phi$ the least-squares misfit function
\begin{equation} \label{eq:misfit}
\Phi(x ; y)=\frac{1}{2}\lVert R^{-\frac{1}{2}}(y-h(x)) \rVert^2=:\frac{1}{2}\lVert y-h(x) \rVert_{R}^2.
\end{equation}
We assume that the posterior distribution (\ref{eq:posterior}) can be represented by a probability density function (PDF) w.r.t.~the Lebesgue 
measure on $\R^{N_x}$, which we denote by $\pi^\ast (x)$, that is, 
\begin{equation}
\pi({\rm d}x|y) = \pi^\ast(x)\,{\rm d}x.
\end{equation}
In the case of a Gaussian prior with mean $\overline{x}_0$ and covariance matrix $P_0$, the posterior can be written as
\begin{equation}
\pi^\ast (x) = \frac{1}{C} \exp(-{\Phi_\cR}(x;y)),
\end{equation}
where ${\Phi_\cR}(x;y):=\Phi(x;y)+\cR(x)$ with ${\cR(x)}:=\frac12 \|x-\overline{x}_0\|_{P_0}^2$. We write for 
simplicity $\Phi(x)$ and ${\Phi_\cR(x)}$, respectively, and ignore the dependence on the data $y$ from now on.

In this paper, we explore time-dependent deterministic as well as stochastic processes of $M$ interacting particles $X_t^{(i)} \in 
\R^{N_x}$, $i=1,\ldots,M$,  with the property that the distribution of the particles $X_t^{(i)}$ approximates $\pi^\ast$ as $t\to \infty$. 
We also assume for simplicity that, at initial time, the particles $X_0^{(i)}$ are independent with their distribution given by the prior $\pi_0$. 
We treat the particle positions $X_t^{(i)}$ as random vectors with values in $\R^{N_x}$, which we collect into the $N_z$-dimensional
random vector
\begin{equation}
Z_t = \left( (X_t^{(1)})^{\rm T},(X_t^{(2)})^{\rm T},\ldots,(X_t^{(M)})^{\rm T}\right)^{\rm T}
\end{equation}
with $N_z = M N_x$, and consider gradient-based evolution equations for BIP of the form
\begin{equation}\label{eq:gradient_methods}
{\rm d}Z_t = -\cA(Z_t)\nabla_z \cV(Z_t)\,{\rm d}t + \Gamma(Z_t) \, {\rm d} W_t,
\end{equation}
$t\ge0$, for some positive semi-definite matrix-valued $\cA(z)\in \mathbb{R}^{N_z\times N_z}$ and $\Gamma (z) \in
\mathbb{R}^{N_z\times N_z}$, $W_t$ standard $N_z$-dimensional Brownian motion, and $\cV:\R^{N_z}\to\R$ 
a given potential. Here $z$ denotes an element in $\R^{N_z}$ of the form
\begin{equation} \label{eq:joint_state}
z = \left( (x^{(1)})^{\rm T},(x^{(2)})^{\rm T},\ldots,(x^{(M)})^{\rm T}\right)^{\rm T}
\end{equation}
with all $x^{(i)}$, $i=1,\ldots,M$, elements of $\R^{N_x}$. Throughout this paper we use the Ito interpretation of the multiplicative noise term in (\ref{eq:gradient_methods}). 

\begin{ex} \label{ex:BD}
The classical example of (\ref{eq:gradient_methods}) is provided by the scaled first-order (overdamped) Langevin (also called Brownian)
dynamics
\begin{equation}\label{eq:BD}
{\rm d}X_t^{(i)} = -\cC \nabla_x{\Phi_{\cR}}(X_t^{(i)})\,{\rm d}t+\sqrt{2}\cC^{1/2}\,{\rm d}W_t^{(i)},
\end{equation}
where $W_t^{(i)}$, $i=1,\ldots,M$, denotes independent $N_x$-dimensional Brownian motion and
$\cC$ is a constant symmetric positive-definite matrix. In this case, the particles do not interact and $\cA = I_M \otimes \cC$. 
Here $A \otimes B$ denotes the Kronecker product of two matrices and $I_M$ the $M$-dimensional identity matrix. Furthermore, 
$\Gamma = \sqrt{2}I_M \otimes \cC^{1/2}$ and the potential $\cV$ is given by
\begin{equation}
\cV(z) = \sum_{i=1}^M {\Phi_\cR}(x^{(i)}).
\end{equation}
It is then well-known that under appropriate conditions on $\cC$ and $\cV$
\begin{equation}
\lim_{t\to \infty} \mbox{Law}\,(X_t^{(i)}) = \pi^\ast, \qquad i=1,\ldots,M,
\end{equation}
in a weak sense \cite{sr:P14}. We assume throughout this paper that these conditions hold for the BIP under consideration, that is, 
$\cV(z) \to \infty$ as $\|z\|\to \infty$, $e^{-\beta\cV(z)}$ is integrable for all $\beta>0$ and $D^2{\Phi_\cR}^2\ge \alpha I$ for some $\alpha>0$.
\end{ex}

The PDF $\pi_t^{(i)}$ for each of the particles $X_t^{(i)}$ in (\ref{eq:BD}), $i=1,\ldots,M$, satisfies the Fokker--Planck equation
\begin{equation} \label{eq:FPE}
\partial_t \pi_t = \nabla_x \cdot \left( \pi_t \cC \nabla_x \frac{\delta {\rm KL}(\pi_t| \pi^\ast)}{\delta \pi_t}\right)
\end{equation}
with $\pi_t = \pi_t^{(i)}$ and  the Kullback--Leibler divergence defined by
\begin{equation} \label{eq:KLD}
{\rm KL}(\pi| \pi^\ast) =  \Bigl \langle \pi,\ln \left(\frac{\pi}{\pi^\ast}\right) \Bigr \rangle ,
\end{equation}
where $\langle g,f\rangle$ denotes the standard $L_2$-inner product, and its variational derivative given by
\begin{equation}
\frac{\delta {\rm KL}(\pi| \pi^\ast)}{\delta \pi} = \ln \left( \frac{\pi}{\pi^\ast} \right) .
\end{equation}
{It is well-known that (\ref{eq:FPE}) can be viewed as a gradient flow in the space of probability measures \cite{sr:JKO98}.}

An interesting generalisation of (\ref{eq:BD}) within the framework of (\ref{eq:gradient_methods}) has recently been proposed
in \cite{AGFHWLAS2019}. It relies on making the matrix $\cC$ dependent on the PDF $\pi_t$ itself. More specifically, 
\begin{equation}\label{eq:theoretical_cov}
\cC(\pi_t) = \E \left[(X_t-\overline{X}_t)(X_t-\overline{X}_t)^{\rm T}\right], \qquad \overline{X}_t = \E\left[X_t\right].
\end{equation}
This choice of $\cC$ is motivated by the gradient flow structure of the ensemble Kalman--Bucy filter as first 
revealed in \cite{sr:BR10} and gives rise to the Kalman--Wasserstein flow structure defined by the associated generalised
Fokker--Planck equation (\ref{eq:FPE}) \cite{AGFHWLAS2019}. See also \cite{SPSR2019,sr:INR19}. 
We will get back to the resulting stochastic interacting particle system in Section \ref{sec:LD}. 
Another generalisation of (\ref{eq:BD}) is provided by the Stein variational gradient descent formulation 
of \cite{sr:LW16} which can be viewed as an interacting particle approximation to (\ref{eq:FPE}) with $\cC = \pi_t I_{N_x}$ 
\cite{sr:LLN18}. { See also \cite{DNS19} for a detailed analysis of Stein variational gradient descent as a gradient
flow in the space of probability measures with its Riemannian geometry determined by $\cC$.}

Note that the Fokker--Planck equation (\ref{eq:FPE}) can be formally viewed as the Liouville equation corresponding to 
the mean-field ordinary differential equation (ODE)
\begin{equation} \label{eq:MODE}
\frac{{\rm d}}{{\rm d}t} X_t = \mathcal{F} (X_t,\pi_t) = -\cC \nabla_x \ln \left( \frac{\pi_t}{\pi^\ast} \right) (X_t).
\end{equation}
This reformulation provides the starting point for the deterministic interacting particle formulations proposed in 
\cite{SR2018,SPSR2019} for BIPs and for the blob method for diffusion in \cite{sr:CCP19}. 
These formulations will be further explored  in Section \ref{sec:FP}. 
The following example provides a particular instance of such an interacting particle system in the context of linear forward models 
and Gaussian distributions.

\begin{ex} \label{ex:Gauss}
Let us assume that the potential ${\Phi_\cR}$ is quadratic and that the number of particles 
is $M \ge N_x+1$. Then one can consider the deterministic interacting particle system
\begin{equation} \label{eq:linearGauss}
\frac{{\rm d}}{{\rm d}t} X_t^{(i)} = -\nabla_x {\Phi_\cR}(X_t^{(i)}) + \left( P_t^{xx}\right)^{-1} (X_t^{(i)} - \overline{X}_t)
\end{equation}
with empirical mean
\begin{equation}
\overline{X}_t = \frac{1}{M} \sum_{i=1}^M X_t^{(i)}
\end{equation}
and empirical covariance matrix
\begin{equation}
P_t^{xx} = \frac{1}{M}\sum_{i=1}^M (X_t^{(i)} - \overline{X}_t)(X_t^{(i)}-\overline{X}_t)^{\rm T}.
\end{equation}
If one choses the initial particle positions $X_0^{(i)}$ such that the associated empirical covariance matrix 
$P_0^{xx}$ is non-singular, then the particle system (\ref{eq:linearGauss}) satisfies
\begin{equation} \label{eq:limitGauss}
\lim_{t\to\infty} \overline{X}_t = \overline{x}^\ast, \qquad \lim_{t\to\infty} P_t^{xx} = P^\ast,
\end{equation}
with $\overline{x}^\ast$ and $P^\ast$ denoting the mean and covariance matrix, respectively, of the Gaussian 
posterior distribution $\pi^\ast$, that is
\begin{equation}
{\Phi_\cR}(x) = \frac{1}{2} (x-\overline{x}^\ast)^{\rm T}(P^\ast)^{-1}(x-\overline{x}^\ast).
\end{equation}
Indeed, it is easily verified that (\ref{eq:linearGauss}) implies
\begin{equation}
\frac{{\rm d}}{{\rm d}t} \overline{X}_t = -(P^\ast)^{-1}(\overline{X}_t - \overline{x}^\ast)
\end{equation}
as well as
\begin{equation}
 \frac{{\rm d}}{{\rm d}t} P_t^{xx} = -(P^\ast)^{-1} P_t^{xx} - P_t^{xx} (P^\ast)^{-1} + 2I_{N_x}.
\end{equation}
Also note that (\ref{eq:linearGauss}) fits into the framework (\ref{eq:gradient_methods}) with $\cA = I_{N_z}$, 
$\Gamma = 0_{N_z}$, and potential
\begin{equation} \label{eq:potl}
\cV(Z_t) = \sum_{i=1}^M {\Phi_\cR}(X_t^{(i)}) - \frac{M}{2} \ln | P_t^{xx}|.
\end{equation}
Alternatively, one could set $\cA = M I_{N_z}$ and scale the potential (\ref{eq:potl}) by $M^{-1}$. This formulation has
the advantage that the resulting potential can be interpreted as an approximation to an expectation value. However,
throughout this paper we will stick to unnormalised potentials of the form (\ref{eq:potl}).
\end{ex}


\section{Computational approaches} \label{sec:computational}


In this section, we first summarise deterministic interacting particle formulations \cite{sr:CCP19,SR2018,SPSR2019}
based on the Wasserstein dynamics in the space of probability measures. These formulations
lead to a finite-dimensional gradient system. We then extend them to 
the recently proposed Kalman--Wasserstein dynamics \cite{AGFHWLAS2019} and contrast the resulting deterministic
formulations with their stochastic counterparts put forward in \cite{AGFHWLAS2019,sr:INR19}. We also
extend the Kalman--Wasserstein dynamics to localised covariance matrices which allows for an efficient and robust
implementation of gradient-free formulations for BIP in the spirit of the ensemble Kalman filter
\cite{Evensen2003,sr:br11,AGFHWLAS2019}. Except for the gradient-free implementations, all presented interacting particle 
formulations fit into the general framework (\ref{eq:gradient_methods}). Such a unifying framework will be advantageous for
exploring hybrid approaches as well as other generalisations of the methods presented in this paper.


\subsection{Fokker--Planck based particle systems} \label{sec:FP}


In this section, we discuss interacting particle approximations to the Fokker--Planck equation (\ref{eq:FPE}) and its associated mean-field ODE
(\ref{eq:MODE}). We start with the case $\cC = I_{N_x}$ and follow the reproducing kernel Hilbert space (RKHS) approach put forward in
\cite{SR2018}. More specifically, given a RKHS $\mathcal{H}$ with symmetric kernel function $k(x,x')$ and inner product 
$\langle g,f\rangle_\mathcal{H}$ we consider the RKHS Kullback--Leibler divergence
\begin{equation} \label{eq:KLH}
{\rm KL}_{\cal H}(\pi|\pi^\ast) := \Bigl \langle \widetilde{\pi},\ln \left( \frac{\widetilde{\pi}}{\pi^\ast}\right) \Bigr \rangle_\mathcal{H} ,
 \end{equation}
with RKHS PDF
 \begin{equation}\label{eq:pi_tilde}
 \widetilde{\pi}(x) = \int_{\mathbb{R}^{N_x}} k(x,x')\pi(x')\,{\rm d}x'  = \langle k(x,\cdot),\pi (\cdot)\rangle.
 \end{equation}
 Here we have assumed that the kernel satisfies {for each $x'\in\R^{N_x}$}
 \begin{equation} \label{eq:normalisation}
 \int_{\R^{N_x}} k(x,x'){\rm d}x = 1.
 \end{equation}
 In fact, it turns out that one can work with unnormalised kernel functions since any normalisation constant vanishes 
 in the variational derivative of (\ref{eq:KLH}). See (\ref{eq:gradient_KL}) below. Hence we will drop condition (\ref{eq:normalisation}) 
 from now on.

\begin{ex}
A popular class of kernel functions is provided by the Gaussian kernels
\begin{equation} \label{eq:Gauss_kernel}
k(x,x') = \psi(\|x-x'\|_B)
\end{equation}
with $\psi(r) = \exp(-r^2/2)$ and $\|x\|_B^2 = x^{\rm T} B^{-1} x$ for some appropriate symmetric positive-definite matrix
$B \in \mathbb{R}^{N_x\times N_x}$. 
One may also consider data-driven kernel functions such as 
\begin{equation} \label{eq:data_kernel}
k(x,x') = \frac{ \psi(\|x-x'\|_B)}{\sqrt{\sum_{i=1}^M \psi(\|X^{(i)}-x'\|_B)}\sqrt{\sum_{j=1}^M \psi(\|x-X^{(j)}\|_B)}},
\end{equation}
which arise from a diffusion map approximation to the semigroup generated by a reversible diffusion process with invariant measure $\pi$
provided that $X^{(i)} \sim \pi$ and $B = 2\epsilon I_{N_x}$ for $\epsilon>0$ sufficiently small \cite{sr:H18,sr:TMM19}. 
\end{ex}
   
The associated RKHS Fokker--Planck equation in $\pi_t$ is now defined by 
\begin{equation} \label{eq:Liouville}
\partial_t \pi_t = -\nabla_x \cdot(\pi_t \mathcal{F})
\end{equation}
with the vector field $\cF$ given by
\begin{equation} \label{eq:RVF}
\cF(x,\pi_t) = -\nabla_x \frac{\delta {\rm KL}_\mathcal{H}(\pi_t|\pi^\ast)}{\delta \pi_t} (x).
\end{equation}
  
\begin{lemma}
The variational derivative of the RKHS Kullback--Leibler divergence is given by
\begin{equation} \label{eq:gradient_KL}
\frac{\delta {\rm KL}_\mathcal{H}(\pi_t|\pi^\ast)}{\delta \pi_t} = \ln \widetilde{\pi}_t - \ln \pi^\ast + \int_{\mathbb{R}^{N_x}} 
k(\cdot,x') \frac{\pi_t(x')}{\widetilde{\pi}_t(x')} {\rm d}x'.
\end{equation}
\end{lemma}

\begin{proof} We first note that the reproducing kernel property
$f(x') = \langle f(\cdot),k(\cdot,x')\rangle_\cH$ implies that
\begin{equation}\label{eq:KLH2}
{\rm KL}_{\cal H}(\pi|\pi^\ast) = \Bigl \langle \pi,\ln \left(\frac{\widetilde{\pi}}{\pi^\ast} \right) \Bigr \rangle
\end{equation}
The lemma then follows from the definition of the variational derivative
\begin{equation}
\Bigl \langle \frac{\delta {\rm KL}_\mathcal{H}(\pi|\pi^\ast)}{\delta \pi} , \delta \pi \Bigr\rangle =
\lim_{\epsilon \to 0} \frac{ {\rm KL}_\mathcal{H}(\pi + \epsilon \delta \pi|\pi^\ast) - {\rm KL}_\mathcal{H}(\pi|\pi^\ast)}{\epsilon}
\end{equation}
and
\begin{subequations}
\begin{align}
\lim_{\epsilon\to 0} \frac{1}{\epsilon}
 \Bigl \langle \pi,  \ln \frac{ \int_{\mathbb{R}^{N_x}} k(\cdot,x') (\pi(x') + \epsilon \delta \pi(x')) {\rm d}x' }{ 
 \int_{\mathbb{R}^{N_x}} k(\cdot,x') \pi(x') {\rm d}x' } \Bigr \rangle &= \Bigl \langle \frac{\pi}{\widetilde{\pi}}, \int_{\mathbb{R}^{N_x}} k(\cdot,x') 
 \delta \pi(x') {\rm d}x'  \Bigr \rangle \\
&= 
\Bigl \langle \int_{\mathbb{R}^{N_x}} 
k(\cdot,x') \frac{\pi(x')}{\widetilde{\pi}(x')} {\rm d}x', \delta \pi \Bigr \rangle 
\end{align}
\end{subequations}
in particular.
\end{proof}

\begin{remark}
Note that $k(x,x') = \delta (x-x')$ leads formally back to the standard definition of the Kullback--Leibler divergence and the second
term in (\ref{eq:gradient_KL}) vanishes. We also note that the blob method proposed in \cite{sr:CCP19}, which relies on 
the regularised Kullback--Leibler divergence
\begin{equation}
{\rm KL}_\epsilon(\pi|\pi^\ast) = \Bigl \langle \pi,\ln \left(\frac{\pi_\epsilon}{\pi^\ast} \right) \Bigr \rangle 
\end{equation}
with mollified PDF
\begin{equation}
\pi_\epsilon (x) = \int_{\R^{N_x}} \phi_\epsilon (x-x') \pi(x){\rm d}x,
\end{equation}
and regularisation parameter $\epsilon >0$ becomes identical to (\ref{eq:Liouville}) for mollification und kernel functions satisfying 
$\phi_\epsilon(x-x') = k(x,x')$. This follows from the equivalence of (\ref{eq:KLH}) and (\ref{eq:KLH2}). We note that the
Gaussian kernel (\ref{eq:Gauss_kernel}) satisfies this property while the data-driven kernel (\ref{eq:data_kernel}) does not.
\end{remark}
 
It follows from (\ref{eq:Liouville}) and (\ref{eq:RVF}) that
\begin{subequations}
\begin{align}
\frac{\rm d}{{\rm d}t} {\rm KL}_\mathcal{H}(\pi_t|\pi^\ast) &= \Bigl \langle \frac{\delta {\rm KL}_\mathcal{H}(\pi_t|\pi^\ast)}{\delta \pi_t} ,\partial_t 
\pi_t\Bigr \rangle\\
& = -\int_{\R^{N_x}} \Bigl \| \nabla_x \frac{\delta {\rm KL}_\mathcal{H}(\pi_t|\pi^\ast)}{\delta \pi_t} \Bigr \|^2 \pi_t  {\rm d}x 
= -\int_{\R^{N_x}} \bigl \| \cF \bigr \|^2 \pi_t  {\rm d}x\\
&\le 0
\end{align}
\end{subequations}
and, according to (\ref{eq:gradient_KL}), critical points $\pi_{\rm c}$ of ${\rm KL}_\mathcal{H}$ satisfy 
\begin{equation}
0=\ln \widetilde{\pi}_{\rm c} - \ln \pi^\ast + \int_{\mathbb{R}^{N_x}} k(\cdot,x') \frac{\pi_{\rm c} (x')}{\widetilde{\pi}_{\rm c}(x')} {\rm d}x' + c,
\end{equation}
where $c$ is a normalisation constant, that is,
\begin{equation} \label{eq:critical}
\pi^\ast (x) \propto \widetilde{\pi}_{\rm c}(x) e^{l(x)}
\end{equation}
with log-likelihood function
\begin{equation}
l(x) := \int_{\mathbb{R}^{N_x}} k(\cdot,x') \frac{\pi_{\rm c} (x')}{\widetilde{\pi}_{\rm c}(x')} {\rm d}x'  .
\end{equation}
In other words, samples $X_{\rm c}^{(i)}$, $i=1,\ldots,M$, from $\widetilde{\pi}_{\rm c}$ can be used to approximate expectation values with respect 
to the target measure $\pi^\ast$ by assigning them importance weights
\begin{equation} \label{eq:IW}
W_c^{(i)} \propto e^{l(X_{\rm c}^{(i)})}.
\end{equation}
The idea is that these weights are more uniform than the importance weights arising from the prior particles $X_0^{(i)}$, $i=1,\ldots,M$, 
and the least-squares misfit function (\ref{eq:misfit}).

A discrete Fokker--Planck particle dynamics is obtained by replacing $\pi_t$ with the empirical measure 
\begin{equation}
\pi_t(x) = \frac{1}{M} \sum_{i=1}^M \delta(x-X_t^{(i)}),
\end{equation}
which leads from (\ref{eq:Liouville}) to
\begin{equation} \label{eq:dynamic_FP1}
\frac{{\rm d}}{{\rm d}t} X_t^{(i)} = F_t(X_t^{(i)})
\end{equation}
with drift term $F_t(x)$ given by
\begin{equation}\label{eq:driftterm}
F_t(x) = - \nabla_x \left\{ \ln \left( \frac{1}{M} \sum_{j=1}^M
k(x,X_t^{(j)}) \right) - \ln \pi^\ast (x)  + \sum_{j=1}^M \frac{k(x,X_t^{(j)})}{\sum_{l=1}^M k(X_t^{(l)},X_t^{(j)})} \right\}.
\end{equation}
The resulting particle dynamics is equivalent to the one derived in \cite{SR2018,SPSR2019} starting from a discrete approximation to the
regularised Kullback--Leibler divergence. The equations  are of gradient flow structure (\ref{eq:gradient_methods}) with potential
\begin{equation} \label{eq:potential_FP}
\mathcal{V}(z) =  \sum_{i=1}^M \left\{ \ln \left( \frac{1}{M} \sum_{j=1}^M k(x^{(i)},x^{(j)}) \right) - \ln \pi^\ast(x^{(i)}) \right\},
\end{equation}
$\mathcal{A} = I_{N_z}$, and $\Gamma \equiv  0_{N_z}$, that is, $F_t(X_t^{(i)}) = -\nabla_{x^{(i)}}\cV(Z_t)$. Also recall that 
\begin{equation} \label{eq:Gaussian_likelihood}
-\ln \pi^\ast (x) = {\Phi_\cR}(x) + C.
\end{equation}
with $C >0$ an appropriate normalisation constant which is irrelevant for the particle dynamics (\ref{eq:dynamic_FP1}).

\begin{remark}
One can expect that (\ref{eq:dynamic_FP1}) converges to the associated RKHS Fokker--Planck dynamics (\ref{eq:Liouville}) 
as the number of particles, $M$, approaches infinity. A rigorous theoretical investigation of the closely related
blob method for diffusion can be found in \cite{sr:CCP19}. See also \cite{sr:R90} and \cite{sr:DM90} for earlier work on
deterministic numerical methods for approximating diffusion processes.  Furthermore, a diffusion map approach has been 
suggested in \cite{sr:TM19} that approximates the $\nabla_x \ln \pi_t$ term in the Fokker--Planck equation (\ref{eq:MODE}) without
an explicit kernel density estimate for $\pi_t$.  While computationally attractive, it is unclear whether such an approximation leads to a 
particle system fitting the gradient flow structure (\ref{eq:gradient_methods}). {We also mention the Stein variational gradient descent
method which, as previously mentioned, provides an interacting particle approximation to  (\ref{eq:MODE}) with a particular choice
of the operator $\cC$. Again, while the continuum limit possesses a gradient flow structure \cite{DNS19}, this structure is lost under
its finite particle approximation.}
\end{remark}

Equation (\ref{eq:critical}) suggests that an equilibrium particle distribution 
\begin{equation}
Z_{\rm c} = \{X_{\rm c}^{(i)}\}_{i=1}^M = \arg \inf_{z\in \R^{N_z}} \cV(z)
\end{equation}
can be used to approximate expectation values with respect to $\pi^\ast$ for $M$ sufficiently large using the following approximation
\begin{equation}\label{eq:approx_limit}
\widehat{\pi}^\ast (x) \propto
\left(\frac1M\sum\limits_{i=1}^M k(x,X_{\rm c}^{(i)})\right)\exp\left(\sum\limits_{j=1}^M\frac{k(x,X_{\rm c}^{(j)})}{\sum\limits_{l=1}^M k(X_{\rm c}^{(l)},X_{\rm c}^{(j)})}\right).
\end{equation}
More specifically, one first collects a desired number of realisations $X^{(i),(k)}$, $k=1,\ldots,K$, from each of the PDFs 
$k(\cdot,X_{\rm c}^{(i)})$, $i=1,\ldots,M$. These $N = M\cdot K$ realisations are then assigned importance weights
\begin{equation} \label{eq:IW_FP}
W^{(i),(k)} \propto \exp\left(\sum\limits_{j=1}^M\frac{k(X^{(i),(k)},X_{\rm c}^{(j)})}{\sum\limits_{l=1}^M k(X_{\rm c}^{(l)},X_{\rm c}^{(j)})}\right).
\end{equation}

The choice of the kernel functions $k(x,x')$ constitutes an important aspect for the computational implementation 
of (\ref{eq:dynamic_FP1}).  We have already mentioned the Gaussian kernel (\ref{eq:Gauss_kernel}) which requires the specification 
of an appropriate covariance matrix $B$. Given $M$ samples $X_0^{(i)}$, $i=1,\ldots,M$, from the prior PDF $\pi_0$ with
corresponding empirical covariance matrix $P_0^{xx}$, we set
\begin{equation} \label{eq:def_B}
B = \alpha P_0^{xx}
\end{equation}
for $\alpha >0$ appropriately chosen. The choice of the bandwidth $\alpha$ is itself a difficult task arising in general in kernel density estimation. We also use the $X_0^{(i)}$'s as initial conditions for (\ref{eq:dynamic_FP1}) in our numerical
experiments. Alternatively, the data-driven kernel (\ref{eq:data_kernel}) can be implemented with $X^{(i)} = X_0^{(i)}$ and $B$ given by (\ref{eq:def_B}). 

\begin{remark}\label{rem:curse_of_dim}
In a high dimensional setting the statistical curse of dimensionality arises, see for example \cite{sw:W06}. If the parameters have dimension $N_x$, 
then we need an ensemble size $M$ growing exponentially fast with $N_x$. The kernel density estimator approximates the target distribution in a local neighbourhood 
of the members in our particle system of size $M$ and in high dimensions those particle locations will be sparse in the parameter space. 
This problem can be counteracted partially by adaptive kernel functions.
Adaptive choices of the kernel functions such as 
\begin{equation} \label{eq:choiceB}
B = \alpha P_t^{xx}
\end{equation}
in the Gaussian kernel (\ref{eq:Gauss_kernel}) can, for example, be considered in the definition of the potential (\ref{eq:potential_FP}) 
and can help to capture the structure of the target distribution. The computation of the associated drift $F_t(X_t^{(i)}) = -\nabla_{x^{(i)}}\cV(Z_t)$ becomes, however, 
more involved. In our high dimensional numerical example we will consider the adaptive choice \eqref{eq:choiceB} but suppress the dependence of $B$ on $x$ in the 
computation of the drift \eqref{eq:driftterm} for simplicity. We note, that this method will no longer be of gradient flow structure \eqref{eq:gradient_methods}, 
but still leads to an effective improvement of the numerical results. 
\end{remark}


\subsection{Preconditioned and gradient-free particle formulations} \label{sec:KW}


We now return to the evolution equations (\ref{eq:gradient_methods}) 
with a non-trivial choice of the matrix $\cA(z)$ and $\Gamma(z)=0_{N_z}$. More specifically,
our choice of $\cA(z)$ is motivated by the gradient flow structure of the ensemble Kalman--Bucy filter \cite{sr:BR10}, that is,
\begin{equation} \label{eq:preconditioned}
\cA(Z_t) = I_M \otimes P_t^{xx} .
\end{equation}
See \cite{SPSR2019} for a more general discussion of preconditioned gradient flows in the context of BIPs. The same preconditioning
has recently been considered for stochastic interacting particle systems with $\Gamma(Z_t) = \sqrt{2} \cA(Z_t)^{1/2}$ in  \cite{AGFHWLAS2019}.
We will provide more details on these formulations in Section \ref{sec:LD}. We also remark that related ideas have previously appeared in
the Markov chain Monte Carlo literature. See, for example, \cite{sr:LMW18} and references therein. 

Applying (\ref{eq:preconditioned}) to the Fokker--Planck particle dynamics of Section \ref{sec:FP}, that is to (\ref{eq:dynamic_FP1}), 
leads to
\begin{equation}\label{eq:dynamic_regFPf}
\frac{{\rm d}}{{\rm d}t} X_t^{(i)} = P_t^{xx} F_t(X_t^{(i)}).
\end{equation}

\begin{ex} Let us return to Example \ref{ex:Gauss}. In the spirit of (\ref{eq:dynamic_regFPf}), we replace (\ref{eq:linearGauss}) by
\begin{subequations}
\begin{align}
\frac{{\rm d}}{{\rm d}t} X_t^{(i)} &= -P_t^{xx} \nabla_x {\Phi_\cR}(X_t^{(i)}) +  X_t^{(i)} - \overline{X}_t\\
&= -P_t^{xx} (P^\ast)^{-1} (X_t^{(i)}-\overline{x}^\ast) + X_t^{(i)} - \overline{X}_t.
\end{align}
\end{subequations}
Because of (\ref{eq:limitGauss}), we obtain
\begin{equation}
P_t^{xx} (P^\ast)^{-1} (X_t^{(i)}-\overline{x}^\ast) \approx X_t^{(i)} - \overline{x}^\ast 
\end{equation}
and, hence,
\begin{equation}
\frac{{\rm d}}{{\rm d}t} X_t^{(i)} \approx -(\overline{X}_t - \overline{x}^\ast)
\end{equation}
for $t\gg 1$ sufficiently large. This suggests that the preconditioning (\ref{eq:preconditioned}) is asymptotically optimal for linear
BIPs, that is, all directions of state space $\R^{N_x}$ and all particles are treated equally as $t\to\infty$. See also Lemma \ref{lemma2}
below on the affine invariance of the preconditioned gradient flow system.
\end{ex}

While (\ref{eq:dynamic_regFPf}) is asymptotically optimal for Gaussian posterior PDFs $\pi^\ast$, this is not necessarily true for
multimodal posterior PDFs in which case one may use a localised covariance matrix, as first considered in \cite{sr:LMW18}.
More specifically, one first defines distance-dependent weights
\begin{equation}\label{eq:weights_loc}
w_t^{ij} = \frac{\exp \left( -\frac{1}{2\gamma} \|X_t^{(i)}-X_t^{(j)}\|_D^2 \right)}{\sum_{l=1}^M \exp \left( - \frac{1}{2\gamma} \|X_t^{(i)}-X_t^{(l)}\|_D^2 \right)},
\end{equation}
where $\gamma >0$ is an appropriate scaling parameter and $D \in \mathbb{R}^{N_x\times N_x}$ an appropriate symmetric 
positive-definite matrix, and then each particle $X_t^{(i)}$ is assigned the weighted covariance matrix
\begin{equation} \label{eq:localised_P}
\begin{split}
P_t^{xx}(X_t^{(i)}) &= \sum_{j=1}^M w_t^{ij} \left( X_t^{(j)} - \overline{X}_t^{(i)}\right)  \left( X_t^{(j)} - \overline{X}_t^{(i)}\right)^{\rm T}\\
		 &= \sum_{j=1}^M w_t^{ij} X_t^{(j)} \left( X_t^{(j)}\right)^{\rm T} - \overline{X}_t^{(i)} \left(\overline{X}_t^{(i)}\right)^{\rm T}
 \end{split}
\end{equation}
with localised mean
\begin{equation}\label{eq:localised_mean}
\overline{X}_t^{(i)} = \sum_{j=1}^M w_t^{ij} X_t^{(j)}.
\end{equation}
Finally, the evolution equations (\ref{eq:dynamic_regFPf}) are replaced by
\begin{equation}\label{eq:dynamic_regFPff}
\frac{{\rm d}}{{\rm d}t} X_t^{(i)} = P_t^{xx}(X_t^{(i)}) F_t(X_t^{(i)})
\end{equation}
for $i=1,\ldots,M$. We note that (\ref{eq:dynamic_regFPff}) also fits into the framework of (\ref{eq:gradient_methods}), where
$\cA(Z_t) \in \mathbb{R}^{N_z\times N_z}$ is block-diagonal with its $i$th block entry given by $P_t^{xx}(X_t^{(i)})$.
Furthermore, 
\begin{equation}
\frac{{\rm d}}{{\rm d}t} \cV(Z_t) \le 0
\end{equation}
along solutions of (\ref{eq:dynamic_regFPff}). Finally, $\gamma \to \infty$ leads to $w_t^{ij} = 1/M$ in (\ref{eq:weights_loc}) 
and the covariance matrix $P_t^{xx}$ is recovered from (\ref{eq:localised_P}). In the sequel we always assume that
\begin{equation} \label{eq:D}
D = P_0^{xx}.
\end{equation}

We now demonstrate that the preconditioned formulations (\ref{eq:dynamic_regFPf}) and (\ref{eq:dynamic_regFPff}),
respectively, are invariant under affine transformations. The importance of affine invariant sampling methods for BIP has 
been highlighted in \cite{sr:GW10}. In the context of our general framework (\ref{eq:gradient_methods}) affine invariance is defined as follows.

\begin{definition}
An SDE (\ref{eq:gradient_methods}) is called invariant under an affine transformation
\begin{equation} \label{eq:transform1}
Z_t = {L} V_t + c,
\end{equation}
$M \in \mathbb{R}^{N_z\times N_z}$ invertible, if the associated equations of motion in $V_t$ are of the form
\begin{equation}
{\rm d}V_t = -\cA(V_t) \nabla_v \cU(V_t){\,{\rm d}t} + \Gamma(V_t)\,{\rm d}W_t
\end{equation}
with the potential $\cU$ defined by 
\begin{equation} \label{eq:transform_p}
\cU(V_t) = \cV({L}V_t+c).
\end{equation}
\end{definition}

We restrict the class of all affine transformations of the form (\ref{eq:transform1}) to those defined component-wise, that is,
\begin{equation} \label{eq:transform2}
X_t^{(i)} = AU_t^{(i)} + b, \qquad V_t = \left( (U_t^{(1)})^{\rm T},(U_t^{(2)})^{\rm T},\ldots,(U_t^{(M)})^{\rm T}\right)^{\rm T},
\end{equation}
$A \in \R^{N_x\times N_x}$ invertible.

\begin{lemma}\label{lemma2}
The preconditioned formulations (\ref{eq:dynamic_regFPf}) and (\ref{eq:dynamic_regFPff}),
respectively, are invariant under affine transformations of the form (\ref{eq:transform2}).
\end{lemma}

\begin{proof}
We note that (\ref{eq:D}) implies that the weights (\ref{eq:weights_loc}) are invariant under (\ref{eq:transform2}).
Furthermore, we obtain
\begin{equation}
P_t^{xx}(X_t^{(i)}) = P_t^{xx}(AU_t^{(i)}+b) = AP_t^{uu}(U_t^{(i)})A^{\rm T}
\end{equation}
and
\begin{equation}
\nabla_{u^{(i)}} \cU(V_t) = A^{\rm T} \nabla_{x^{(i)}} \cV(Z_t)
\end{equation}
with $\cU$ defined by (\ref{eq:transform_p}). The invariance property follows now from
\begin{equation}
{\frac{\rm d}{{\rm d}t}}X_t^{(i)} = A\,{\frac{\rm d}{{\rm d}t}}U_t^{(i)} = -P_t^{xx}(X_t^{(i)}) \nabla_{x^{(i)}}\cV(Z_t) = -A\left\{P_t^{uu}(U_t^{(i)}) \nabla_{u^{(i)}}\cU(V_t)\right\}.
\end{equation}
\end{proof} 

\begin{remark}
The choice (\ref{eq:def_B}) for the Gaussian kernels (\ref{eq:Gauss_kernel}) 
implies that the transformed potential (\ref{eq:transform_p}) is given by
\begin{equation}
\cU(v) = \sum_{i=1}^M \left\{ \ln \left( \frac{1}{M} \sum_{j=1}^M k(v^{(i)},v^{(j)}) \right) - \ln \pi^\ast(Av^{(i)}+b) \right\}.
\end{equation}
\end{remark}

One of the attractive features of the ensemble Kalman filter \cite{sr:evensen,sr:stuart15,sr:reichcotter15} is its gradient-free formulation
for posterior PDFs satisfying (\ref{eq:Gaussian_likelihood}). In order to extend this approach to our preconditioned 
gradient flow formulations we note that
\begin{subequations}
\begin{align}
P_t^{xx} \nabla_x \ln \pi^\ast (X_t^{(i)}) &= -P_t^{xx} \nabla_x {\Phi_\cR}(X_t^{(i)}) \\
&= -P_t^{xx} Dh(x)^{\rm T} R^{-1} (h(X_t^{(i)})-y) - P_t^{xx} P_0^{-1}(X_t^{(i)}
- \overline{x}_0)).
\end{align}
\end{subequations}
The key idea of the gradient-free formulations of the ensemble Kalman filter \cite{sr:evensen,sr:reichcotter15} is to replace 
$P_t^{xx} Dh(x)^{\rm T}$ with the covariance matrix  
\begin{equation}
P_t^{xh} = \frac{1}{M}\sum_{i=1}^M (X_t^{(i)}-\overline{X}_t)(h(X_t^{(i)})-\overline{h}_t)^{\rm T}, \qquad \overline{h}_t = \frac{1}{M}
\sum_{i=1}^M h(X_t^{(i)}).
\end{equation}
We emphasise that 
\begin{equation} \label{eq:cov_relation}
P_t^{xx} Dh(x)^{\rm T} = P_t^{xh}
\end{equation}
for linear forward maps $h$. For nonlinear forward maps $h$ this approximation will get more accurate if the particles are close to each other. Since the 
particles are representing a distribution, a localised version of the covariance matrix suggests to improve the accuracy of the gradient-free formulation. Increasing the ensemble size will also improve the accuracy of the gradient-free formulation, in particular in the localised formulation.

Following the corresponding continuous-time formulations of the ensemble Kalman--Bucy filter 
\cite{sr:br11,SPSR2019}, we obtain the following gradient-free reformulation of (\ref{eq:dynamic_regFPf}):
\begin{subequations}\label{eq:dynamic_regFPf_nonlinear_2}
\begin{align}
\frac{\rm d}{{\rm d}t} X_t^{(i)} &= - P_t^{xx}\nabla_{x^{(i)}}\left\{\ln \left( \frac1M\sum_{j=1}^M
k(X_t^{(i)},X_t^{(j)}) \right) + \sum_{j=1}^M \frac{k(X_t^{(i)},X_t^{(j)})}{\sum_{l=1}^M k(X_t^{(l)},X_t^{(j)})} \right\} \\
& \qquad - \left\{P_t^{xh} R^{-1}(h(X_t^{(i)})-y) + P_t^{xx}P_0^{-1}(X_t^{(i)}-\bar x_0) \right\}.
\end{align}
\end{subequations}
While no longer of gradient flow structure, the potential (\ref{eq:potential_FP}) still allows us to monitor the 
behaviour of (\ref{eq:dynamic_regFPf_nonlinear_2}) in the large time limit.
 
We note that (\ref{eq:cov_relation}) also holds for our localised covariance matrices $P_t^{xx}(X_t^{(i)})$ with $P_t^{xh}(X_t^{(i)})$ defined by
\begin{equation}
P_t^{xh}(X_t^{(i)}) = \sum_{{j=1}}^M w_t^{ij}(X_t^{(j)} - \overline{X}_t^{(i)})(h(X_t^{(j)})-\overline{h}_t^{(i)})^{\rm T}, \qquad
\overline{h}^{(i)}_t = \sum_{{j=1}}^M w_t^{ij} h(X_t^{(j)}).
\end{equation}
Hence, the localised formulation (\ref{eq:dynamic_regFPff}) gives rise to the gradient-free formulation
\begin{equation}\label{eq:dynamic_regFPf_nonlinear_3}
\begin{split}
\frac{\rm d}{{\rm d}t} X_t^{(i)} &= - P_t^{xx}(X_t^{(i)})\nabla_{x^{(i)}}\left\{\ln \left(\frac1M \sum_{j=1}^M
k(X_t^{(i)},X_t^{(j)}) \right) + \sum_{j=1}^M \frac{k(X_t^{(i)},X_t^{(j)})}{\sum_{l=1}^M k(X_t^{(l)},X_t^{(j)})} \right\} \\
& \qquad - \left\{P_t^{xh}(X_t^{(i)}) R^{-1}(h(X_t^{(i)})-y) + P_t^{xx} (X_t^{(i)}) P_0^{-1}(X_t^{(i)}-\bar x_0) \right\}.
\end{split}
\end{equation}
We will demonstrate in Section \ref{sec:numerics} that the localised formulation (\ref{eq:dynamic_regFPf_nonlinear_3}) 
is beneficial in case of multimodal posterior PDFs $\pi^\ast$.

\begin{remark} We emphasise that the localisation strategy proposed here is different from standard B-localisation
employed for ensemble Kalman filters, where the empirical covariance matrix $P_t^{xx}$ is tempered by a second matrix
$C$ such that the preconditioning matrix becomes $C\circ P_t^{xx}$ where $\circ$ denotes the Schur product of two matrices. 
See, for example, \cite{sr:evensen,sr:reichcotter15}. Furthermore, the proposed localised strategy can also be applied in the context of ensemble Kalman inversion \cite{DBCSPWSW2019, StLawIg2013, NKAS2018, SchSt2016}.
\end{remark}


\subsection{Langevin dynamics based particle systems} \label{sec:LD}

In \cite{AGFHWLAS2019}, the authors proposed the interacting Langevin dynamics 
\begin{equation}\label{eq:interacting_langevin}
{\rm d}X_t^{(i)} = -P_t^{xx}\nabla_x{\Phi_\cR}(X_t^{(i)})\,{\rm d}t+\sqrt{2}(P_t^{xx})^{1/2}\,{\rm d}{W_t^{(i)}},
\end{equation}
where $W_t^{(i)},\ i=1,\dots,M$, denote independent $N_x$-dimensional Brownian motions. The large particle limit $M\to\infty$ leads formally to 
the mean-field equation
\begin{equation}
{\rm d}X_t = -\mathcal{C}(\pi_t)\nabla_x{\Phi_\cR}(X_t)+\sqrt{2}\mathcal{C}^{1/2}(\pi_t)\,{\rm d}W_t,
\end{equation}
with $\cC(\pi)$ defined in \eqref{eq:theoretical_cov}, where the corresponding marginal densities $\pi_t,\ t\ge0,$ evolve according to the 
nonlinear Fokker--Planck equation
\begin{equation} \label{eq:nonlinear_FPE}
\partial_t \pi_t = \nabla_x \cdot\left(\pi_t\mathcal{C}(\pi_t)\nabla_x \frac{\delta{\rm KL}(\pi_t|\pi^\ast)}{\delta \pi_t}\right).
\end{equation}
The mathematical properties of (\ref{eq:nonlinear_FPE}) have been studied \cite{AGFHWLAS2019}. 
We note that a corresponding nonlinear Fokker--Planck equation
arises formally from the large ensemble limit of (\ref{eq:dynamic_regFPf}) with ${\rm KL}(\pi|\pi^\ast)$ replaced by the
RKHS Kullback--Leibler divergence (\ref{eq:KLH}).

Furthermore, in \cite{sr:NR19}, the authors propose a corrected finite-size particle system  in order to obtain the 
correct long-time behaviour even under finite ensemble sizes. More specifically, the corrected particle system evolves according to the system of SDEs
\begin{equation} \label{eq:corrected_BD}
{\rm d}X_t^{(i)} = -P_t^{xx}\nabla_x {\Phi_\cR}(X_t^{(i)})\,{\rm d}t+\frac{N_x+1}{M}(X_t^{(i)}-\overline{X}_t)\,{\rm d}t+\sqrt{2}(P_t^{xx})^{1/2}\,{\rm d}{W_t^{(i)}},
\end{equation}
where the gradient descent direction in the drift term has been corrected through the expression
\begin{equation}\label{eq:correction_term}
\nabla_{x^{(i)}} \cdot P_t^{xx} = \frac{N_x+1}{M}(X_t^{(i)}-\overline{X}_t)\in\R^{N_x}.
\end{equation}
Again we note that (\ref{eq:corrected_BD}) fits within the general framework of (\ref{eq:gradient_methods}) with
$\cA(Z_t) = I_M \otimes P_t^{xx}$, $\Gamma (Z_t) = \sqrt{2} I_M \otimes (P_t^{xx})^{1/2}$ and potential
\begin{equation}
\cV(Z_t) = \sum_{i=1}^M {\Phi_\cR}(X_t^{(i)}) - \frac{N_x+1}{2} \ln | P_t^{xx}|.
\end{equation}
Here we have assumed that $P_t^{xx}$ has full rank and used that 
\begin{equation}
\frac{\partial}{\partial (P_t^{xx})_{ij}} \ln |P_t^{xx}| = \left((P_t^{xx})^{-1}\right)_{ij},
\end{equation}
where $(A)_{ij}$ denotes the $(i,j)$th entry of a matrix $A$, and, hence,
\begin{equation}
\frac{1}{2} \nabla_{x^{(i)}} \ln |P_t^{xx}| = \frac{1}{M} (P_t^{xx})^{-1} (X_t^{(i)}-\overline{X}_t).
\end{equation}

As demonstrated in \cite{sr:NR19}, this correction leads to the Fokker--Planck equation
\begin{equation}
\partial_t \rho_t =\nabla_z\cdot\left(\rho_t \cA \nabla_z\frac{\delta{\rm KL}(\rho_t|\rho^\ast)}
{\delta \rho_t}\right),
\end{equation}
for the marginal PDF $\rho_t(z)$ in the state variable (\ref{eq:joint_state}) and the asymptotic behaviour
\begin{equation}
\lim_{t\to \infty} \rho_t = \rho^\ast
\end{equation}
follows under appropriate conditions on the potential ${\Phi_\cR}$. Here $\rho^\ast(z):= \prod_{i=1}^M \pi^\ast(x^{(i)})$. 
Hence formulation (\ref{eq:corrected_BD}) provides the appropriate generalisation of (\ref{eq:BD})
under the state-dependent diffusion matrix $\Gamma(Z_t)$ and finite ensemble sizes $M$.

A detailed theoretical study, including its affine invariance, of the formulation (\ref{eq:corrected_BD}) as well as 
efficient numerical implementations avoiding the need for computing the square root of the empirical covariance
matrix $P_t^{xx}$ can be found in \cite{sr:INR19}. 

We now derive a correction term for the particle evolution with $P_t^{xx}$ 
in \eqref{eq:interacting_langevin} replaced by the localised empirical covariance matrix \eqref{eq:localised_P}. 
Following \cite{sr:NR19}, the correction term is given by $\nabla_{x^{(i)}} \cdot P_t^{xx}(X_t^{(i)})$ and an explicit
expression is provided in the following lemma.

\begin{lemma}
It holds that
\begin{subequations} \label{eq:correction_local}
\begin{align}
\nabla_{x^{(i)}} \cdot P_t^{xx}(X_t^{(i)}) &= w_t^{ii}(N_x+1)(X_t^{(i)}-\overline{X}_t^{(i)}) + \sum_{j=1}^M X_t^{(j)} \left(X_t^{(j)}\right)^{\rm T} \nabla_{x^{(i)}} w_t^{ij}\\
&\qquad - \,\, \sum_{j=1}^M \overline{X}_t^{(i)} \left(X_t^{(j)}\right)^{\rm T} \nabla_{x^{(i)}}w_t^{ij} - 
 \sum_{j=1}^M X_t^{(j)} \left(\overline{X}_t^{(i)}\right)^{\rm T} \nabla_{x^{(i)}}w_t^{ij},
\end{align}
\end{subequations}
where $\bar X_t^{(i)}\in\R^{N_x}$ denotes the localised mean defined in \eqref{eq:localised_mean} and $w_t^{ij}\in[0,1]$ denote the 
localisation weights (\ref{eq:weights_loc}).
\end{lemma}

\begin{proof}
We will make use of the following identities for the computation of the correction term:
\begin{subequations}
\begin{align}
\nabla_{x^{(i)}} \cdot (w_t^{ii} x^{(i)} (x^{(i)})^{\rm T} ) &= w_t^{ii}(N_x+1)x^{(i)} + x^{(i)} (x^{(i)})^{\rm T} \nabla_{x^{(i)}} w_t^{ii}\\
\nabla_{x^{(i)}} \cdot (w_t^{ij} x^{(i)} (x^{(j)})^{\rm T} ) &= w_t^{ij}x^{(j)}+ x^{(i)} (x^{(j)})^{\rm T} \nabla_{x^{(i)}} w_t^{ij}\\
\nabla_{x^{(i)}} \cdot (w_t^{ij} x^{(j)} (x^{(i)})^{\rm T} ) &= w_t^{ij} N_x x^{(j)} + x^{(j)} (x^{(i)})^{\rm T} \nabla_{x^{(i)}} w_t^{ij}\\
\nabla_{x^{(i)}} \cdot (w_t^{ij} x^{(j)} (x^{(j)})^{\rm T} ) &= x^{(j)} (x^{(j)})^{\rm T} \nabla_{x^{(i)}} w_t^{ij}.
\end{align}
\end{subequations}
Here we have assumed that $j\not=i$.
Applying these identities to the localised covariance matrix (\ref{eq:localised_P}) yields 
\begin{equation}
\nabla_{x^{(i)}} \cdot \left( \sum_{j=1}^M w_t^{ij} X_t^{(j)}\left( X_t^{(j)}\right)^{\rm T} \right) = w_t^{ii}(N_x+1)X_t^{(i)} + \sum_{j=1}^M
X_t^{(j)}\left( X_t^{(j)}\right)^{\rm T}\nabla_{x^{(i)}} w_t^{ij}
\end{equation}
as well as 
\begin{align}
\nabla_{x^{(i)}} \cdot \left( \overline{X}_t^{(i)} \left( \overline{X}_t^{(i)}\right)^{\rm T} \right) &= w_t^{ii}(N_x+1)\overline{X}_t^{(i)} \\ &+
\sum_{j=1}^M \left\{\overline{X}_t^{(i)}\left(X_t^{(j)}\right)^{\rm T} + X_t^{(j)} \left(\overline{X}_t^{(i)}\right)^{\rm T} \right\} \nabla_{x^{(i)}}w_t^{ij}
\end{align}
and (\ref{eq:correction_local}) follows.
\end{proof}

The correction term requires the computation of $\nabla_{x^{(i)}} w_t^{ij}$ for weights given by (\ref{eq:weights_loc}):
\begin{subequations}
\begin{align}
\nabla_{x^{(i)}} w_t^{ij} &= \frac{w_t^{ij}}{\gamma} D^{-1} \biggl\{ -(X_t^{(i)}-X_t^{(j)}) \\ &\qquad\qquad\quad\ \ + \frac{\sum_{l=1}^M \exp\left( -\frac{1}{2\gamma}\|X_t^{(i)}-X_t^{(l)}\|_D^2\right) 
(X_t^{(i)}-X_t^{(l)})}{\sum_{l=1}^M \exp \left(-\frac{1}{2\gamma} \|X_t^{(i)}-X_t^{(l)}\|_D^2\right)}
\biggr\}\\
&= \frac{w_t^{ij}}{\gamma} D^{-1} \biggl\{ -(X_t^{(i)}-X_t^{(j)}) + \sum_{l=1}^M w_t^{il} (X_t^{(i)}-X_t^{(l)})\biggr\} \\
&= \frac{w_t^{ij}}{\gamma} D^{-1} \left\{X_t^{(j)} - \overline{X}_t^{(i)}\right\}.
\end{align}
\end{subequations}

With this result we can formulate the corrected evolution system for the interacting Langevin diffusion model with localised covariance matrix by extending the dynamical system with the drift $\nabla_{x^{(i)}}\cdot P_t^{xx}(X_t^{(i)})$, that is,
\begin{equation}\label{eq:interacting_langevin_corrected}
{\rm d}X_t^{(i)} = -P_t^{xx}(X_t^{(i)})\nabla{\Phi_\cR}(X_t^{(i)})\,{\rm d}t+\nabla_{x^{(i)}} \cdot P_t^{xx}(X_t^{(i)})\,{\rm d}t+\sqrt{2}(P_t^{xx}(X_t^{(i)}))^{1/2}\,{\rm d}W_t^{{{(i)}}}.
\end{equation}
In analogy to (\ref{eq:corrected_BD}), this leads to the following lemma for the joint density of the particle system evolving through \eqref{eq:interacting_langevin_corrected}. See \cite{sr:NR19}.

\begin{lemma}
The joint density corresponding to the particle system evolving by
 \eqref{eq:interacting_langevin_corrected} satisfies the Fokker--Planck equation given by
\begin{equation}
\partial_t\rho_t = \nabla_z\cdot \left( \rho_t \cA \nabla_z\frac{\delta{\rm KL}(\rho_t|\rho^\ast)}{\delta \rho_t}\right)
\end{equation}
with $\cA \in \R^{N_z\times N_z}$ a block diagonal matrix with its $i$th block entry given by $P_t^{xx}(X_t^{(i)})$.
\end{lemma}

Hence by the Kalman--Wasserstein gradient flow structure \cite{AGFHWLAS2019,sr:NR19} in the joint distribution $\rho_t$, 
the finite-size particle system \eqref{eq:interacting_langevin_corrected} can be used in the long-time limit to approximate i.i.d.~samples 
from $\pi^\ast$.


Gradient-free variants of both formulations (\ref{eq:corrected_BD}) 
and (\ref{eq:interacting_langevin_corrected}), respectively, can now be formulated in a straightforward
manner by replacing $P_t^{xx} Dh(X_t^{(i)})^T$ by $P_t^{xh}$ \cite{AGFHWLAS2019} 
or $P_t^{xh}(X_t^{(i)})$ respectively. For example, we obtain the following gradient-free formulation
\begin{subequations}\label{eq:interacting_langevin_corrected_gradientfree}
\begin{align}
{\rm d}X_t^{(i)} &= -\left\{P_t^{xh}R^{-1}(h(X_t^{(i)})-y) + P_t^{xx}P_0^{-1}(X_t^{(i)}-\overline{x}_0) \right\}{\rm d}t\\
&\qquad \qquad 
+\,\nabla_{x^{(i)}} \cdot P_t^{xx}\,{\rm d}t+\sqrt{2}(P_t^{xx})^{1/2}\,{\rm d}W_t^{(i)}
\end{align}
\end{subequations}
and the corresponding localised gradient-free formulation
\begin{subequations}\label{eq:interacting_langevin_corrected_localised}
\begin{align}
{\rm d}X_t^{(i)} &= -\left\{P_t^{xh}(X_t^{(i)})R^{-1}(h(X_t^{(i)})-y) + P_t^{xx}(X_t^{(i)})P_0^{-1}(X_t^{(i)}-\overline{x}_0) \right\}{\rm d}t\\
&\qquad \qquad 
+\,\nabla_{x^{(i)}} \cdot P_t^{xx}(X_t^{(i)})\,{\rm d}t+\sqrt{2}(P_t^{xx}(X_t^{(i)}))^{1/2}\,{\rm d}{W_t^{(i)}}.
\end{align}
\end{subequations}


\section{Numerical results} \label{sec:numerics}

In this section, we apply the proposed methodologies to a sequence of increasingly more challenging numerical examples ranging
from low to high-dimensional and unimodal to multimodal. We have collected the code for our numerical examples in the GitHub repository \cite{sw:SR2021}.

\subsection{2-dimensional unimodal example}\label{subsec:2d}

Before we discuss a bimodal example to show the improving effects of the localised gradient-free formulation, 
we implement a example, which is nearly Gaussian, originally presented in \cite{ErnstEtAl2015} and later also used in 
\cite{AGFHWLAS2019, MHGV2018}. 

More specifically, we consider the following one-dimensional elliptic boundary-value problem
\begin{equation*}
-\frac{{\rm d}}{{\rm d}s}\left(\exp(x_1)\frac{{\rm d}}{{\rm d}s}p(s)\right) = 1,\quad s\in[0,1],
\end{equation*}
with boundary condition $p(0) = 0$ and $p(1) = x_2$. An explicit solution of this boundary-value problem in the parameters 
$x = (x_1,x_2)^{\rm T} \in \mathbb{R}^2$ is given by
\begin{equation*}
p(s,x) = x_2s+\exp(-x_1)\left(-\frac{s^2}{2}+\frac{s}2\right)
\end{equation*}
and the forward map in \eqref{eq:IP} is defined by 
\begin{equation*}
h(x) = (p(s_1,x),p(s_2,x))^{\rm T},
\end{equation*}
that is, we assume that noisy measurements of $p(\cdot,x)$ are given at locations $s_1=0.25$ and $s_2=0.75$.

Furthermore, we assume Gaussian measurement errors $\Xi\sim\cN(0,R)$ with $R=0.01\cdot I_2$, $I_2\in\R^{2\times2}$ the identity matrix,
and a Gaussian prior $\cN(0,P_0)$ with $P_0=100\cdot I_2$. The  data is constructed by drawing a reference parameter 
$x^\dagger\sim\pi_0$ and by setting the observation to $y = h(x^\dagger)+\xi^\dagger$, where $\xi^\dagger$ is a realisation of the 
measurement error. Our numerical results are based on the realisations $x^\dagger = (0.0865, -0.8157)^{\rm T}$ and 
$y = (-0.0173,-0.573)^{\rm T}$.

We use the MATLAB solver {\rm ode45} to solve the ODE representing the deterministic Fokker--Planck dynamics 
\eqref{eq:dynamic_regFPf_nonlinear_2} and the Euler--Maruyama scheme with a step-size $\Delta t = 0.0001$ for 
the interacting Langevin sampler \eqref{eq:interacting_langevin_corrected_localised}. The preconditioned Fokker--Planck dynamics 
is implemented using Gaussian kernels (\ref{eq:Gauss_kernel}) with $B = \alpha P_0$ and 
$\alpha = 0.05$. {Both methods have been initialized by an i.i.d.~sample from the prior distribution $N(0,P_0)$.}

Figure \ref{fig:2d_Stuart_particles} shows the posterior approximation through the particle systems with $M=200$ particles resulting 
from the deterministic Fokker--Planck dynamics as well as the interacting Langevin sampler. One finds that both methods approximate 
the posterior distribution well. 

\begin{figure}[!htb]
	\includegraphics[width=1\textwidth]{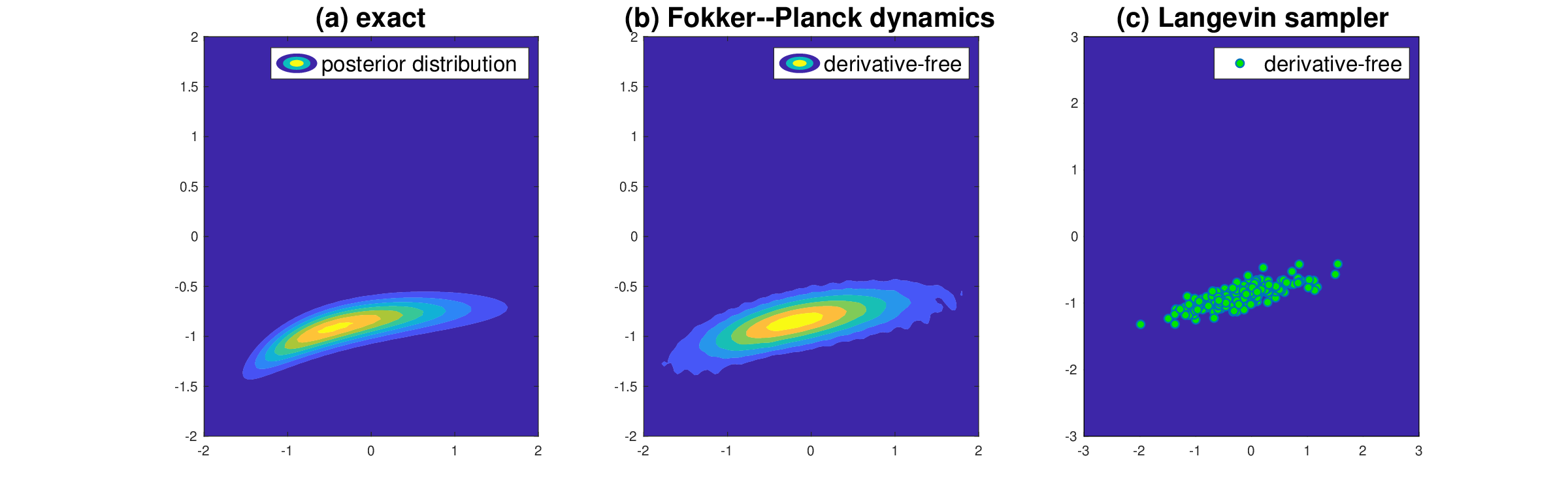}
    	\caption{Approximations to the posterior PDF from interacting Fokker--Planck dynamics: b) kernel density estimate 
	from deterministic Fokker--Planck dynamics, c) particle system from interacting Langevin sampler. The exact posterior PDF is displayed in 
	panel a).}\label{fig:2d_Stuart_particles}
\end{figure}

The performance of the deterministic Fokker--Planck dynamics depends crucially on the kernel parameter $\alpha$ in (\ref{eq:def_B}) as
can be seen from Figure \ref{fig:2d_Stuart_estimate_different_alpha}, where kernel density estimates for the target distribution for different choices 
of $\alpha$ are displayed. The effect of the different choices of $\alpha$ is also demonstrated in the time evolution of the potential 
$\cV$ from \eqref{eq:potential_FP}, which can be found in Figure \ref{fig:2d_Stuart_potential_different_alpha}. If the scaling 
$\alpha$ is chosen too small, then the resulting density underestimates the spread, whereas a too large $\alpha$ overestimates the spread. 
Overall, the choice of the scaling parameter $\alpha$ is crucial and quite sensitive, which is a general challenge in kernel density estimation.

\begin{figure}[!htb]
	\includegraphics[width=1\textwidth]{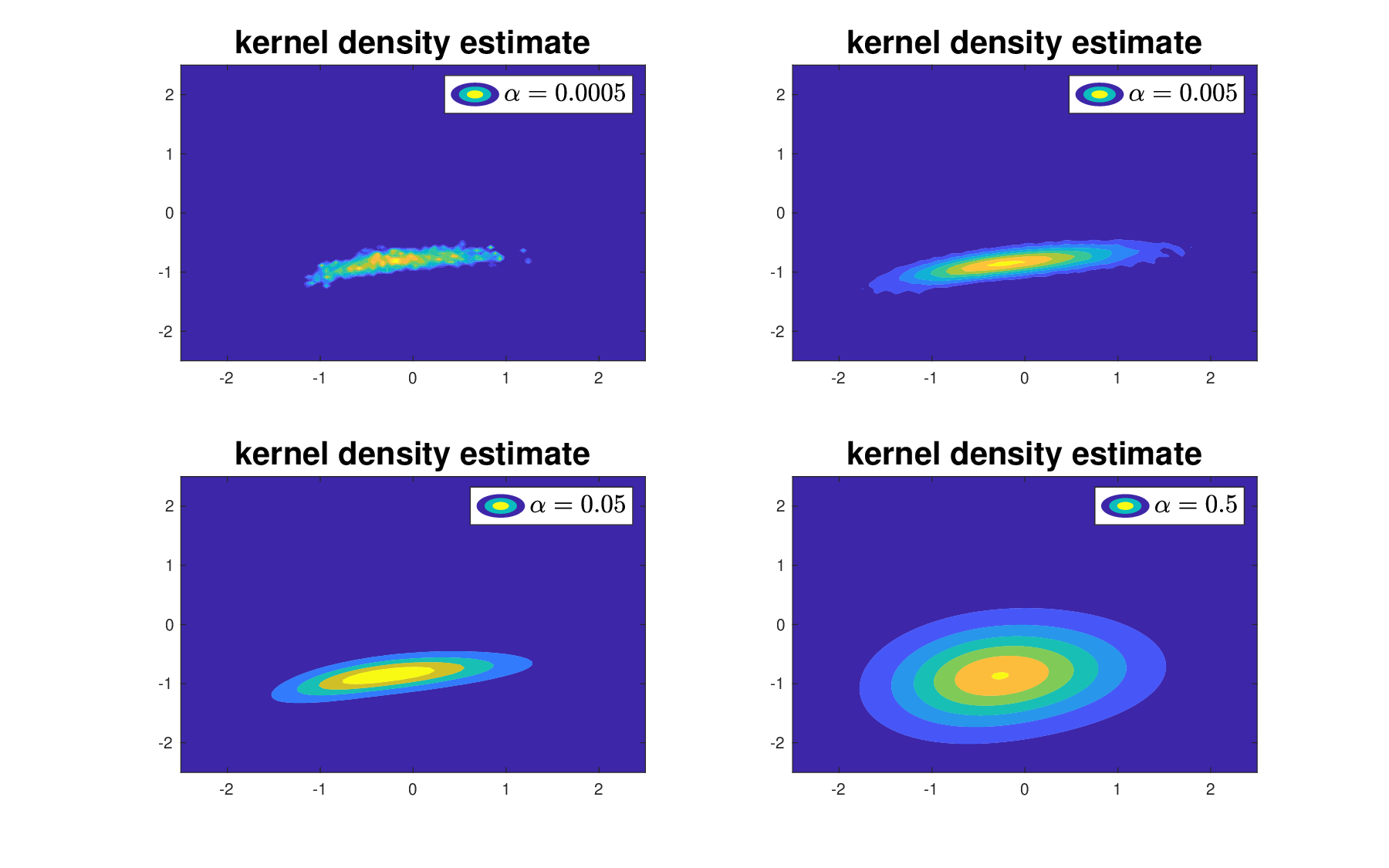}
    	\caption{Approximations to the posterior PDF from deterministic Fokker--Planck dynamics for different values of
	the kernel parameter $\alpha$: a) $\alpha=0.0005$, b) $\alpha=0.005$, c) $\alpha=0.05$, d) $\alpha=0.5$.}\label{fig:2d_Stuart_estimate_different_alpha}
\end{figure}
\begin{figure}[!htb]
	\includegraphics[width=1\textwidth]{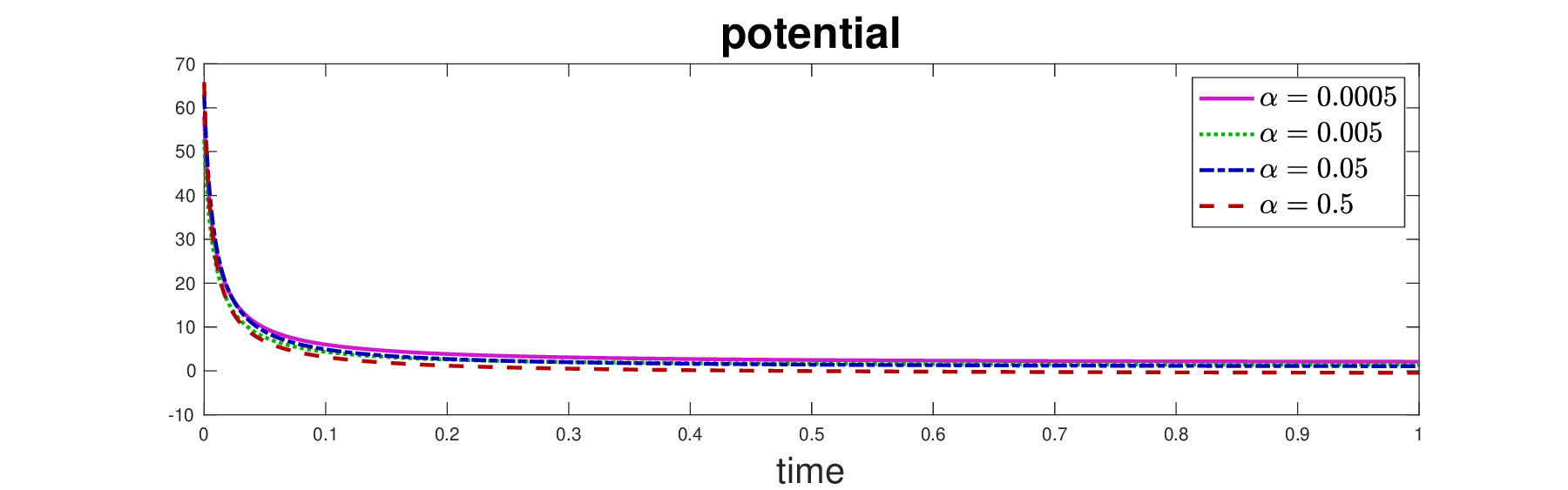}
    	\caption{Time evolution of the potential function $\cV$ for different choices of $\alpha$.}\label{fig:2d_Stuart_potential_different_alpha}
\end{figure}

We find that both the deterministic and stochastic interacting particle formulations work well for this simple 2-dimensional example. While
the interacting Langevin dynamics is easier to implement, the Fokker--Planck dynamics immediately results in a kernel density estimate for
the posterior distribution and its performance can be monitored through the time evolution of the potential energy \eqref{eq:potential_FP}.

\subsection{2-dimensional bimodal example}\label{subsec:2d_bimodal}

We next consider a 2-dimensional bimodal example resulting from the nonlinear forward map
\begin{equation}
h:\R^2\to\R, \quad h(x) = (x_1-x_2)^2.
\end{equation}
We assume a Gaussian prior with mean zero and covariance $P_0 = I_2\in\R^{2\times 2}$ and Gaussian measurements errors 
$\Xi \sim\cN(0,R)$ with $R = I_2$.  We draw a reference parameter $x^\dagger\sim\pi_0$ and an 
observation $y = h(x^\dagger)+\xi^\dagger$, where $\xi^\dagger$ is a realisation of the random measurement errors $\Xi$. 
Our numerical results are based on the realisation $x^\dagger=(-1.5621,-0.0021)^{\rm T}$ 
and $y=4.2297$.

We implement the preconditioned version of the Fokker--Planck based particle system {\eqref{eq:dynamic_FP1}}. Furthermore, we 
also consider its gradient-free formulation {\eqref{eq:dynamic_regFPf_nonlinear_2}} as well as the localised 
formulation {\eqref{eq:dynamic_regFPf_nonlinear_3}}. We compare the results to those from the corresponding interacting
Langevin sampler {\eqref{eq:corrected_BD}}, its {gradient-free formulation \eqref{eq:interacting_langevin_corrected_gradientfree}}, and 
its localised formulation {\eqref{eq:interacting_langevin_corrected}}, respectively.

We again use the MATLAB solver {\rm ode45} to time-step the deterministic Fokker--Planck dynamics \eqref{eq:dynamic_regFPf_nonlinear_2} 
and the Euler--Maruyama scheme with step-size $\Delta t = 0.0001$ for the interacting Langevin sampler \eqref{eq:interacting_langevin_corrected_localised}.  We use a Gaussian kernel (\ref{eq:Gauss_kernel}) with $B = \alpha P_0$ and $\alpha = 0.01$
for the preconditioned Fokker--Planck dynamics. The weights for the localised formulation 
\eqref{eq:weights_loc} are computed with $\gamma=0.5$ and $D= P_0\in\R^{2\times 2}$. All simulations use $M = 200$ particles {initialized by an i.i.d.~sample from the prior distribution $N(0,P_0)$.}


Figure \ref{fig:2d_2hills_FPPS_densities} 
shows the kernel density estimates resulting from the preconditioned Fokker--Planck dynamics, which demonstrates that the particle system is representing the target density. Furthermore, while the gradient-free formulation \eqref{eq:dynamic_regFPf_nonlinear_2} is getting pushed 
to one of the peaks, the localised formulation \eqref{eq:dynamic_regFPf_nonlinear_3} leads to an effectively improved approximation of the 
posterior density. 
The time evolution of the potential \eqref{eq:potential_FP} along the three different interacting particle approximations is displayed
in Figure \ref{fig:2d_2hills_FPPS_Vt}.

\begin{figure}[!htb]
	\begin{center}
	\includegraphics[width=1\textwidth]{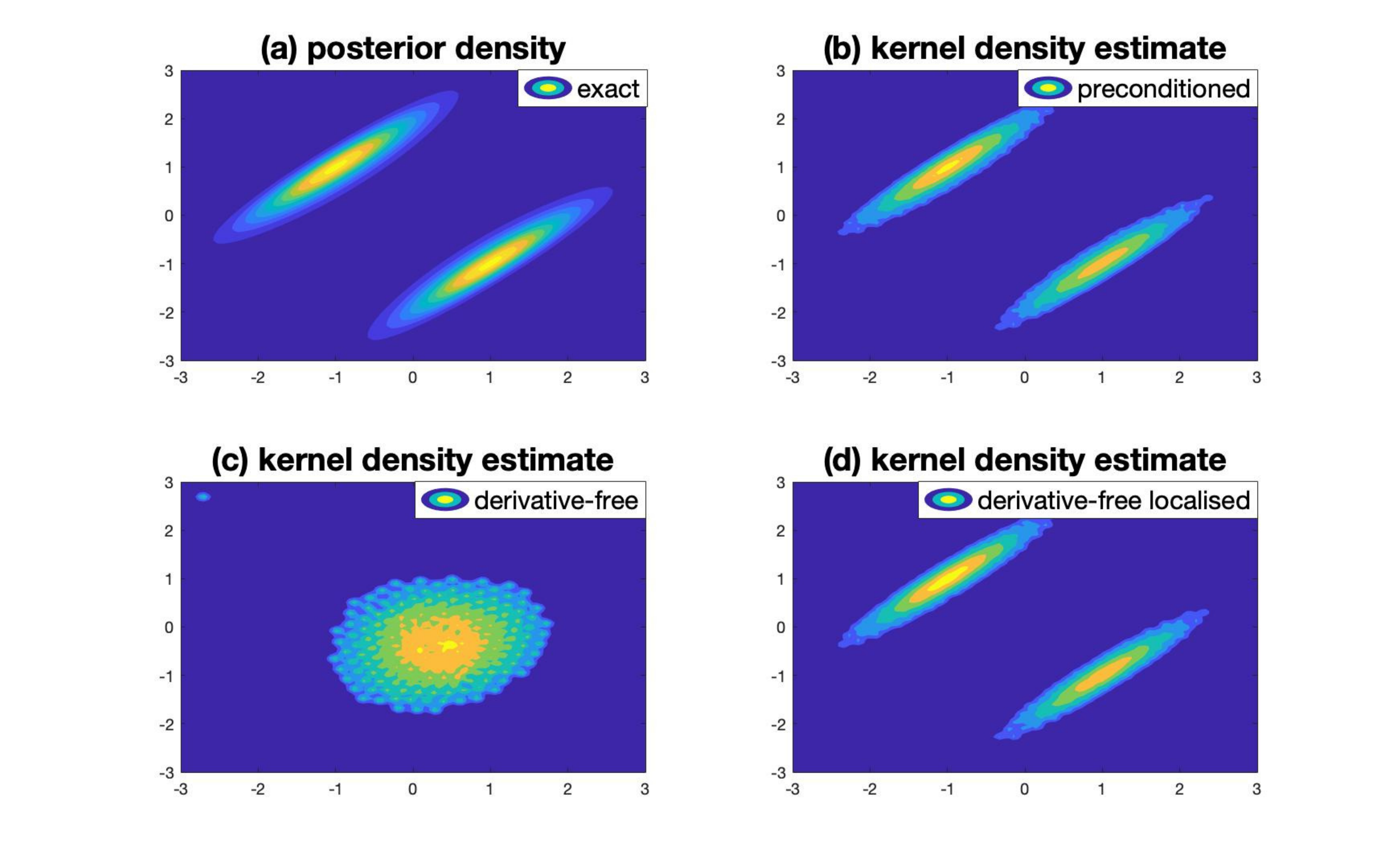}
	 \end{center}
    	\caption{Approximations to the posterior PDF from interacting Fokker--Planck dynamics: a) exact posterior PDF, b) fitted PDF from
	preconditioned dynamics using exact gradients, c) fitted PDF from preconditioned gradient-free dynamics, d) fitted PDF from localised gradient-free dynamics.}\label{fig:2d_2hills_FPPS_densities}
\end{figure}

\begin{figure}[!htb]
	\begin{center}
	\includegraphics[width=1\textwidth]{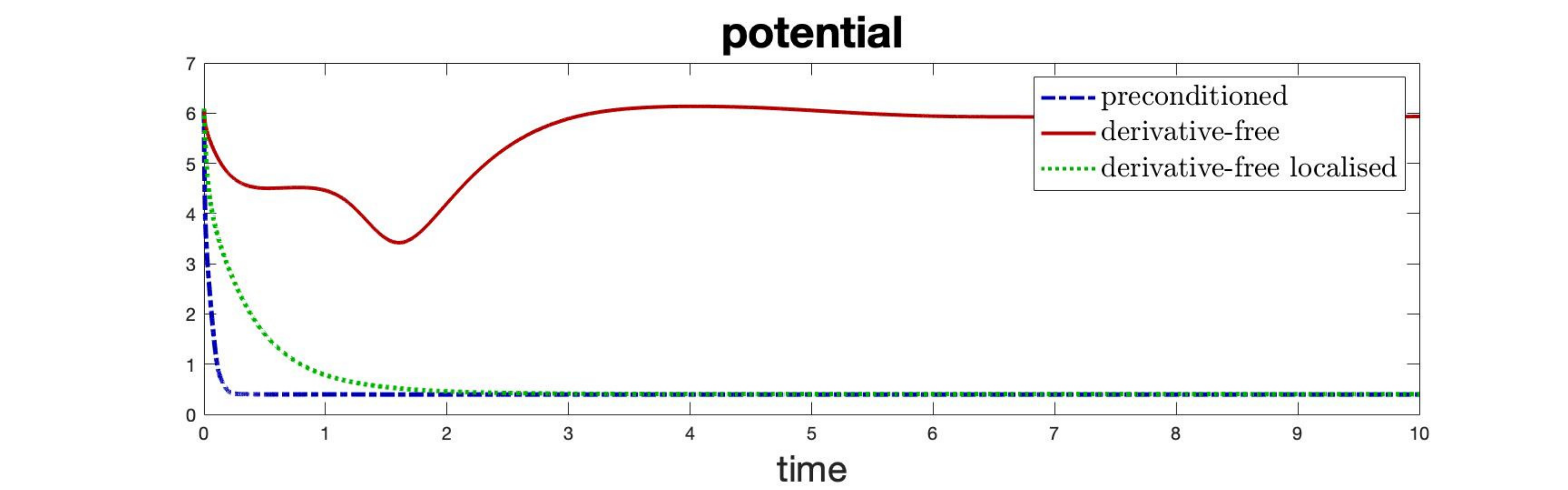}
    	\caption{Time evolution of potential function $\cV$ for different implementations of Fokker--Planck particle dynamics: (blue) preconditioned dynamics using exact gradients, (red) preconditioned gradient-free dynamics, (green) localised gradient-free dynamics.}\label{fig:2d_2hills_FPPS_Vt}
    \end{center}
\end{figure}

We present the variational derivative of the RKHS Kullback--Leibler divergence \eqref{eq:gradient_KL} resulting from the equilibrium particle positions $\{X_c^{(i)}\}$ as well as the weight function 
\begin{equation}
W(x) = \exp\left(\sum\limits_{j=1}^M\frac{k(x,X_{\rm c}^{(j)})}{\sum\limits_{l=1}^M k(X_{\rm c}^{(l)}X_{\rm c}^{(j)})}\right)
\end{equation}
(compare \eqref{eq:approx_limit}) corresponding to the equilibrium particle positions in Figure \ref{fig:2d_2hills_FPPS_var_derivative}.
We find that both the exact gradient method as well as the localised gradient-free formulation are leading to a variational derivative, 
which is nearly constant in the region of state space covered by the equilibrium particle positions $\{X_c^{(i)}\}$.

\begin{figure}[!htb]
	\begin{center}
	\includegraphics[width=1\textwidth]{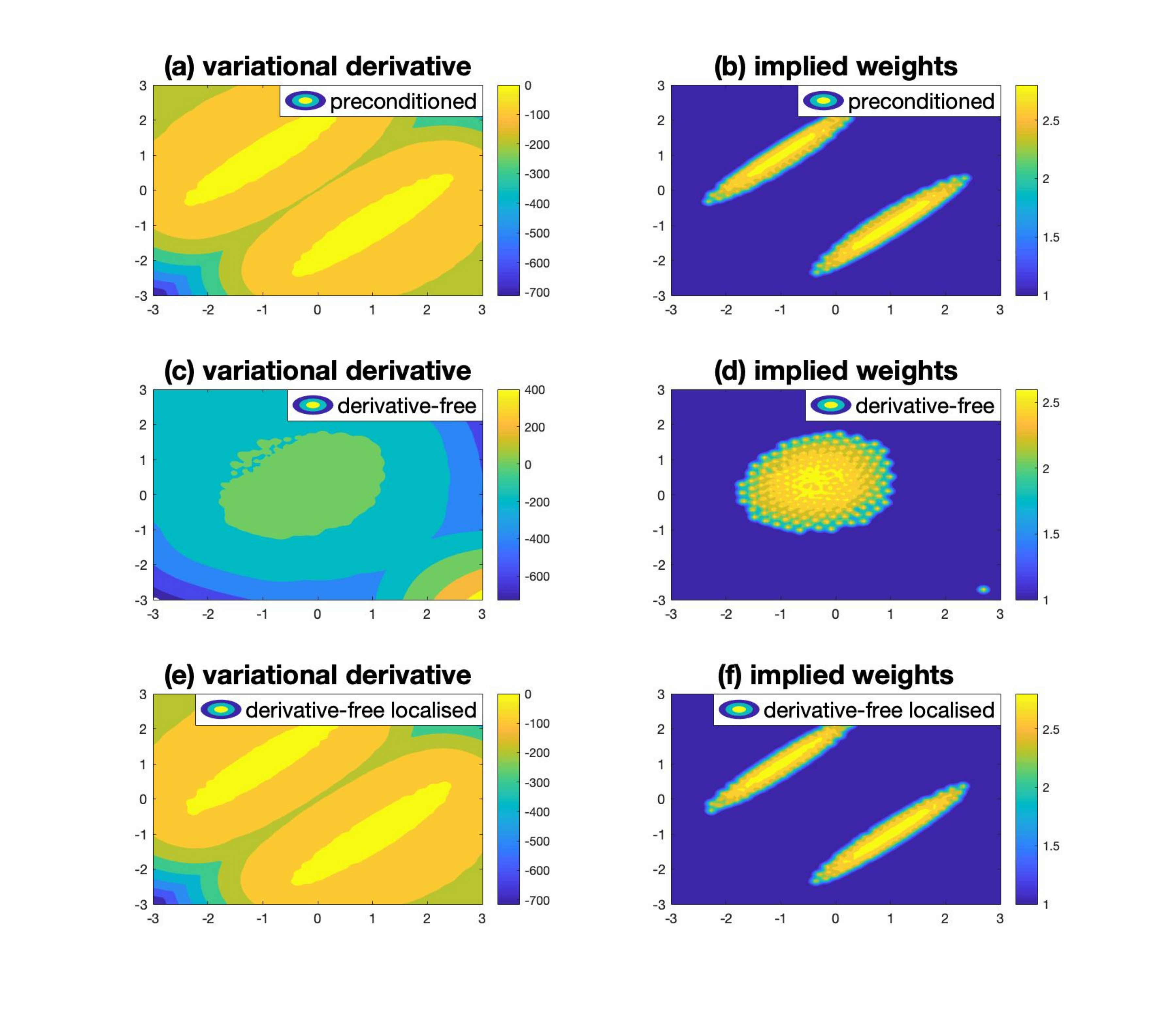}
    	\caption{Variational derivative (left) and implied weights (right) of the RKHS Kullback--Leibler divergence for different implementations of Fokker--Planck particle dynamics: (a)-(b) preconditioned dynamics using exact gradients, (c)-(d) preconditioned gradient-free dynamics, (e)-(f) localised gradient-free dynamics.}\label{fig:2d_2hills_FPPS_var_derivative}
    \end{center}
\end{figure}

Qualitatively similar results are also obtained from the corresponding implementations of the interacting Langevin dynamics. See Figure \ref{fig:2d_2hills_EKS_particles}. The gradient-free formulation shows again a focus into the direction of one of the peaks, 
while the localisation highly improves the approximation of the posterior distribution.

\begin{figure}[!htb]
	\includegraphics[width=1\textwidth]{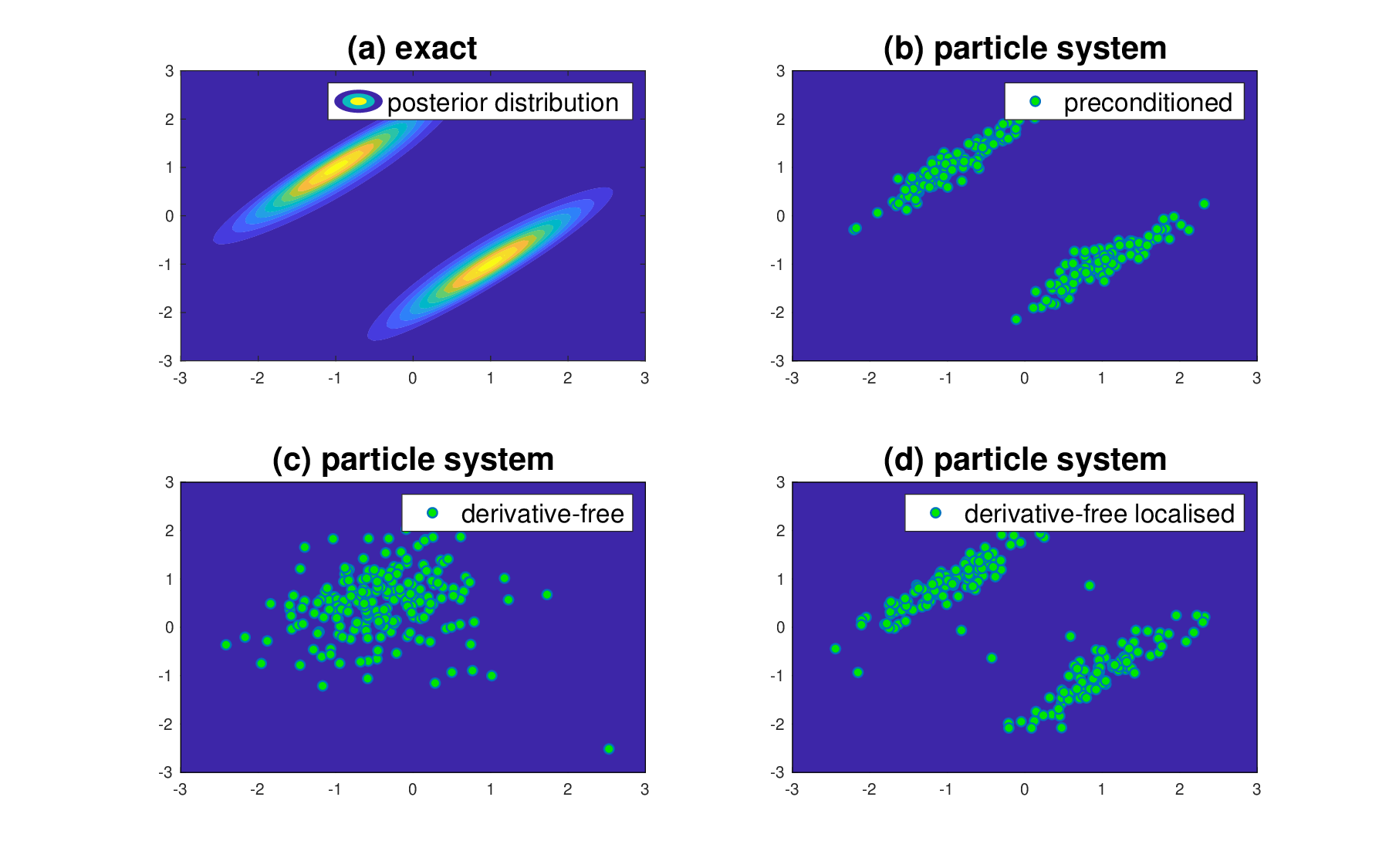}
    \caption{Approximations to the posterior PDF from interacting Langevin dynamics: a) exact posterior PDF, b) particle system from
	preconditioned dynamics using exact gradients, c) particle system from preconditioned gradient-free dynamics, d) particle system from localised gradient-free dynamics.}\label{fig:2d_2hills_EKS_particles}
\end{figure}

Figure \ref{fig:2d_2hills_estimate_different_gamma} shows the kernel density estimate resulting from the localised deterministic Fokker--Planck dynamics \eqref{eq:dynamic_regFPf_nonlinear_3} for different choices of the localisation scaling $\gamma>0$. We see that for too small choices of $\gamma$ the localisation effect is too strong and the particles move too slowly. The reason for this is that the particles do not find nearby particles to interact with. If,
on the other hand, the scaling parameter $\gamma$ is chosen too large, the localisation effect is too weak and the particles start 
concentrating on one of the two peaks. This effect results from the inaccurate approximation of the gradient of the forward problem. Figure \ref{fig:2d_2hills_potential_different_gamma} shows the time evolution of the potential function  \eqref{eq:potential_FP} 
for different scaling localisation parameter $\gamma$. We notice that the estimate for $\gamma=10$ is leading to a lower value of the 
potential compared to the estimate for $\gamma=0.1$ with $\gamma = 1$ being optimal. While the estimate for $\gamma = 10$ represents one of the peaks nearly 
perfectly while ignoring the second peak, the two other estimates approximates both peaks fairly well. {In summary, the parameter $\gamma$ scales the interaction between the particles through preconditioning. For nonlinear problems a low value of $\gamma$ is needed, as the particles have to be close to each other for a good derivative approximation. However, choosing a low value for $\gamma$ requires a higher number of particles $M$ since some particles might have only minor interactions with other particles resulting in a poor approximation of derivatives if $M$ is too small. Hence, one should choose $\gamma$ in a suitable way such that it decreases with $M$, e.g.~$\gamma(M) = c\cdot M^{-\delta}$ for $\delta>0$, where $\delta$ depends on the nonlinearity of the forward model and the dimension of the parameter space.}


\begin{figure}[!htb]
	\includegraphics[width=1\textwidth]{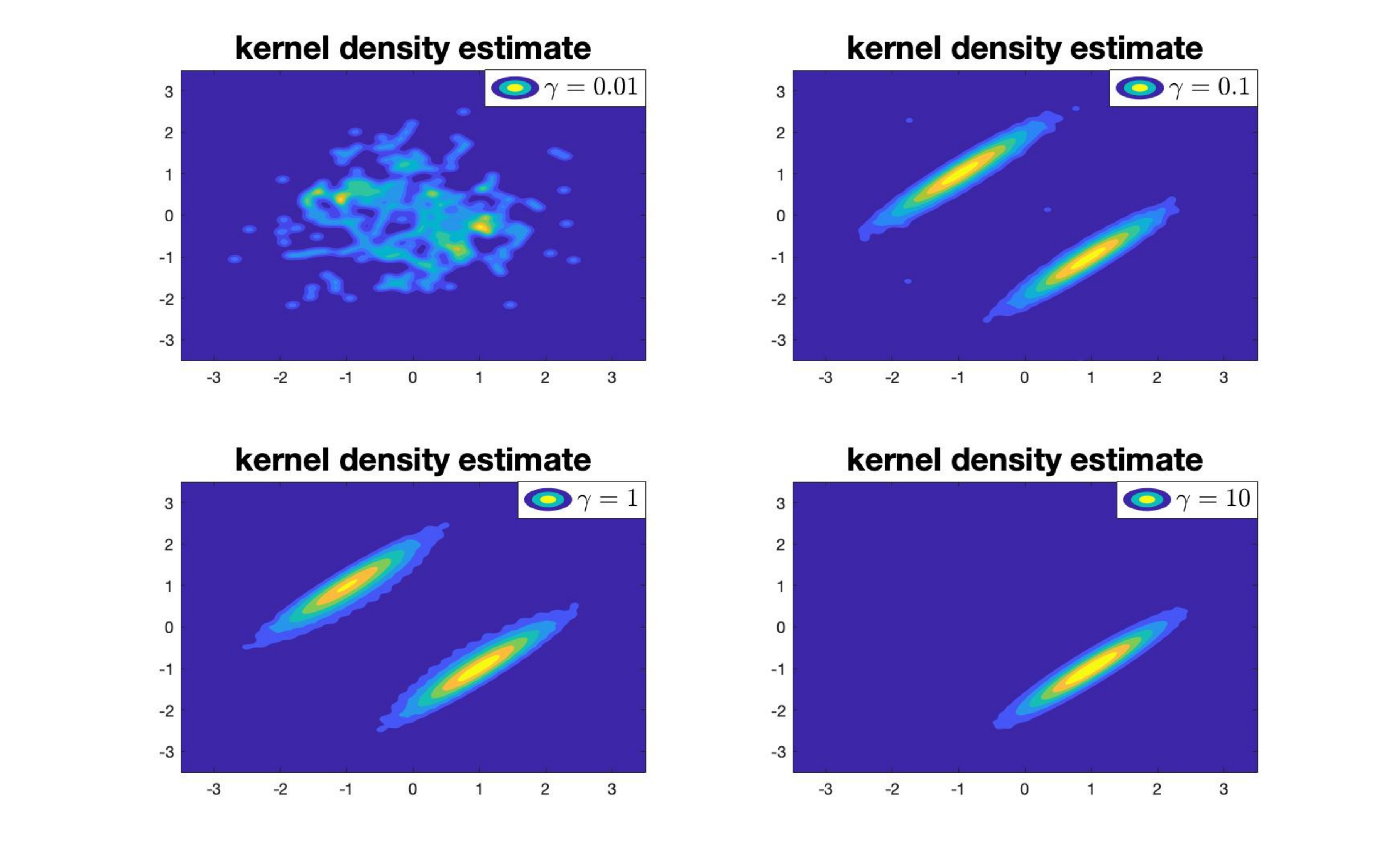}
    	\caption{Approximations to the posterior PDF from interacting Fokker--Planck dynamics: a) $\gamma=0.1$, b) $\gamma=1$, c) $\gamma=4$, d) $\gamma=10$.}\label{fig:2d_2hills_estimate_different_gamma}
\end{figure}
\begin{figure}[!htb]
	\includegraphics[width=1\textwidth]{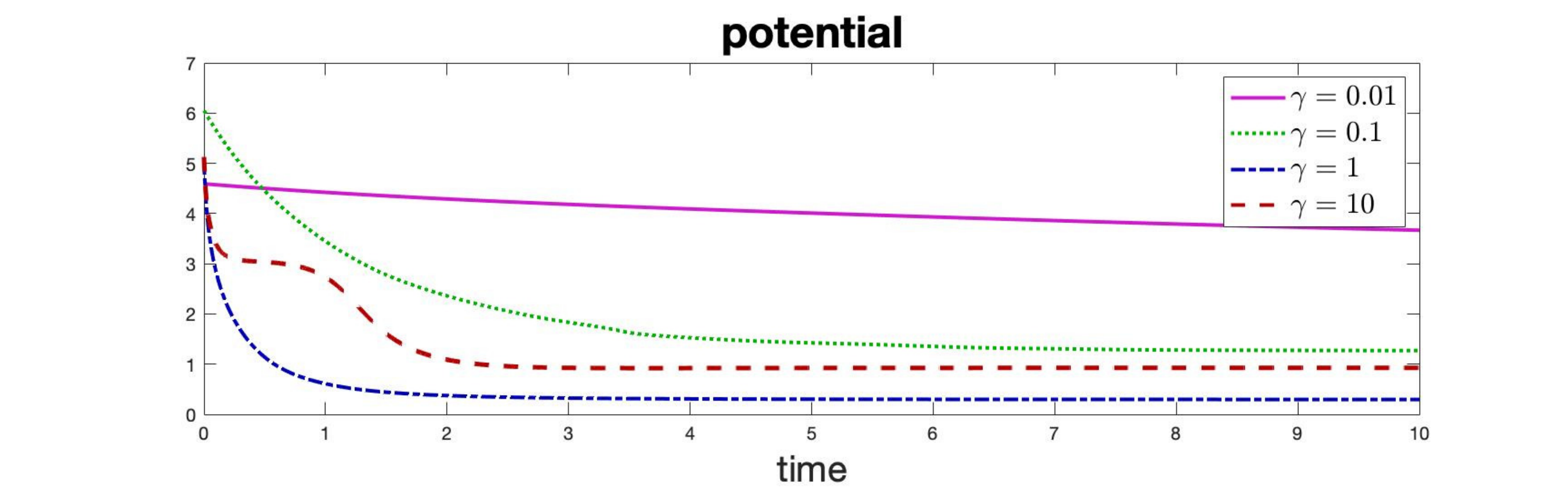}
    	\caption{Time evolution of the potential function $\cV$ for different choices of $\gamma$.}\label{fig:2d_2hills_potential_different_gamma}
\end{figure}

In summary, this example demonstrates the effectiveness of the localised gradient-free formulations both in its deterministic and stochastic
form. We also note that a careful choice of the parameter $\gamma$ is required and that smaller values of $\gamma$ require larger 
ensemble sizes $M$ in order to provide faithful gradient approximations.

\subsection{Scalability in high dimensions}\label{subsec:dimensionalscaling}

{
In the following we consider a simple toy example in order to study the behavior of the deterministic Fokker--Planck particle system \eqref{eq:dynamic_regFPf_nonlinear_2}  based on the RKHS approach. To do so, we consider a Gaussian process $GP(0,(-\Delta)^\tau)$ represented by the (truncated) KL expansion
\begin{equation}\label{eq:KL_expansion}
u(s,x) = \sum\limits_{k=1}^{N_x} x_k \psi_k(s),
\end{equation}
where $\Delta$ denotes the Laplacian operator over $\mathcal D = [0,1]$ equipped with Dirichlet boundary conditions, $x=(x_1,\dots,x_{N_x})^{\rm T}\in\R^{N_x}$ is a vector of independent Gaussian distributed random variables $\mathcal N(0,\lambda_k)$ with $\lambda_k=k^{-2\tau}$ and $\psi_k(s)=\sqrt{2\pi}\sin(2\pi s)$, \cite{sr:S15}. We consider the inverse problem of recovering the coefficients $x\in\R^{N_x}$ of one observed sampled path of the Gaussian process. The KL expansion will be truncated at index $N_x$, and the state space of the Gaussian process will be discretized on a uniform grid $\mathcal D_l\subset[0,1]$ with mesh size $h=2^{-l}$, this is the KL expansion will be evaluated at $N_y=2^l$ points. For fixed $N_x$ and fixed $N_y$, we can write the problem as linear inverse problem in the form of \eqref{eq:IP} where the forward model is defined by $h(\cdot) = A\cdot$, where the $k-th$ column of $A = A_{N_x,N_y}\in\R^{N_y\times N_x}$ consists of $(\psi_k(s_1), \dots,\psi(s_{N_y})^{\rm T}$, $s_i = i\cdot h$, $i=1,\dots N_y$. We set a prior to $X_0\sim\mathcal N(0,P_0)$ with $P_0\in\R^{N_x\times N_x}$ being a diagonal matrix with entries $\lambda_k,\ k=1,\dots,N_x$ and $\tau=1$ and assume Gaussian noise $N(0,R)$ with $R=I_{N_y}$. The reference data $y$ will be constructed by drawing $x_k^\dagger\sim\mathcal N(0,\lambda_k)$ and computing $y_{N_x,N_y}=A_{N_x,N_y}x^\dagger$ with $x^\dagger = (x_1^\dagger,\dots,x_{N_x}^\dagger)^{\rm T}$. }
 
 {
We will analyze the performance of the deterministic Fokker--Planck dynamics \eqref{eq:dynamic_regFPf_nonlinear_2} for increasing the dimension in both $N_x$ and $N_y$.  We consider two settings: 
\begin{itemize}
\item In the first setting we keep the truncation of the KL expansion fixed to $N_x=4$ and increase the dimension $N_y$ of the observations by choosing $N_y\in\{16,64,256\}$.
\item In the second setting we increase the truncation of the KL expansion by choosing the index $N_x\in\{4,6,8\}$ and keeping the dimension of the observations fixed to $N_y = 64$.
\end{itemize}
}

{
In our numerical results, we initialize the particle system by an i.i.d. sample of $N(0,P_0)$ and solve
\begin{equation}
\frac{\rm d}{{\rm d}t}X_t^{(i)} = P_t^{xx}\cdot F_t(X_t^{(i)})
\end{equation}
for $F$ defined by \eqref{eq:driftterm}, where we consider Gaussian kernels \eqref{eq:Gauss_kernel} with different choices of $B$. In particular, we will choose $B$ such that it scales with the dimension $N_x$ and the ensemble size $M$ in order to obtain good approximation results i.e., we test $$B = \frac{c_\delta}{M^{\delta}}\diag((P_0)_{ii}),\ B_\ast = \frac{c_\delta}{M^{\delta}}\diag((P_\ast)_{ii})\ \text{and}\ B_t = \frac{c_\delta}{M^{\delta}}\diag((P_t^{xx})_{ii}),$$ 
where we define
\begin{equation} \label{eq:bandwdth}
\delta = \frac{1}{N_x+4}\quad\text{and}\quad c_\delta = \left(\frac{4}{N_x+2}\right)^\delta.
\end{equation}
These choices of kernels correspond to the product of univariate kernels in each dimensions with optimal bandwidth for Gaussian kernels minimizing the asymptotic mean integrated squared error \cite[Section~6.3.1]{sw:S92}. While in the choice of $B_\ast$ we assume to have access to the theoretical variance $\sigma_i^2 = (P_\ast)_{ii}$ of each component, we are using approximation $\widehat\sigma_i^2 = (P_0)_{ii}^2$ for the choice $B$ and $\widehat\sigma_i^2 = (P_t^{xx})_{ii}$ for the choice $B_t$ respectively.
}

{Testing these choices of kernels, we observe that for $B$ the approximation results are getting worse if the prior covariance $P_0$ is far away from the target covariance $P_\ast$ as $\widehat\sigma_i^2= (P_0)_{ii}^2$ is overestimating the variance of the posterior. For the choice $B_\ast$ we obtain high accurate approximation results, however, in practical situations it is an infeasible choice as the true covariance is typically unknown. The third choice $B_t$ corresponds to the adaptive kernel choice introduced in Remark~\ref{rem:curse_of_dim}. This choice helps by approximating $P_\ast$ through the particle system and updating the underlying RKHS adaptively. In our numerical results, we observe for increasing dimension $N_x$ but fixed observations $N_y$ that the variance $\sigma_i$ in each component gets underestimated as the particle system collapses, such that the resulting approximation fails for increasing $N_x$.
}


\begin{figure}[!htb]
	\includegraphics[width=1\textwidth]{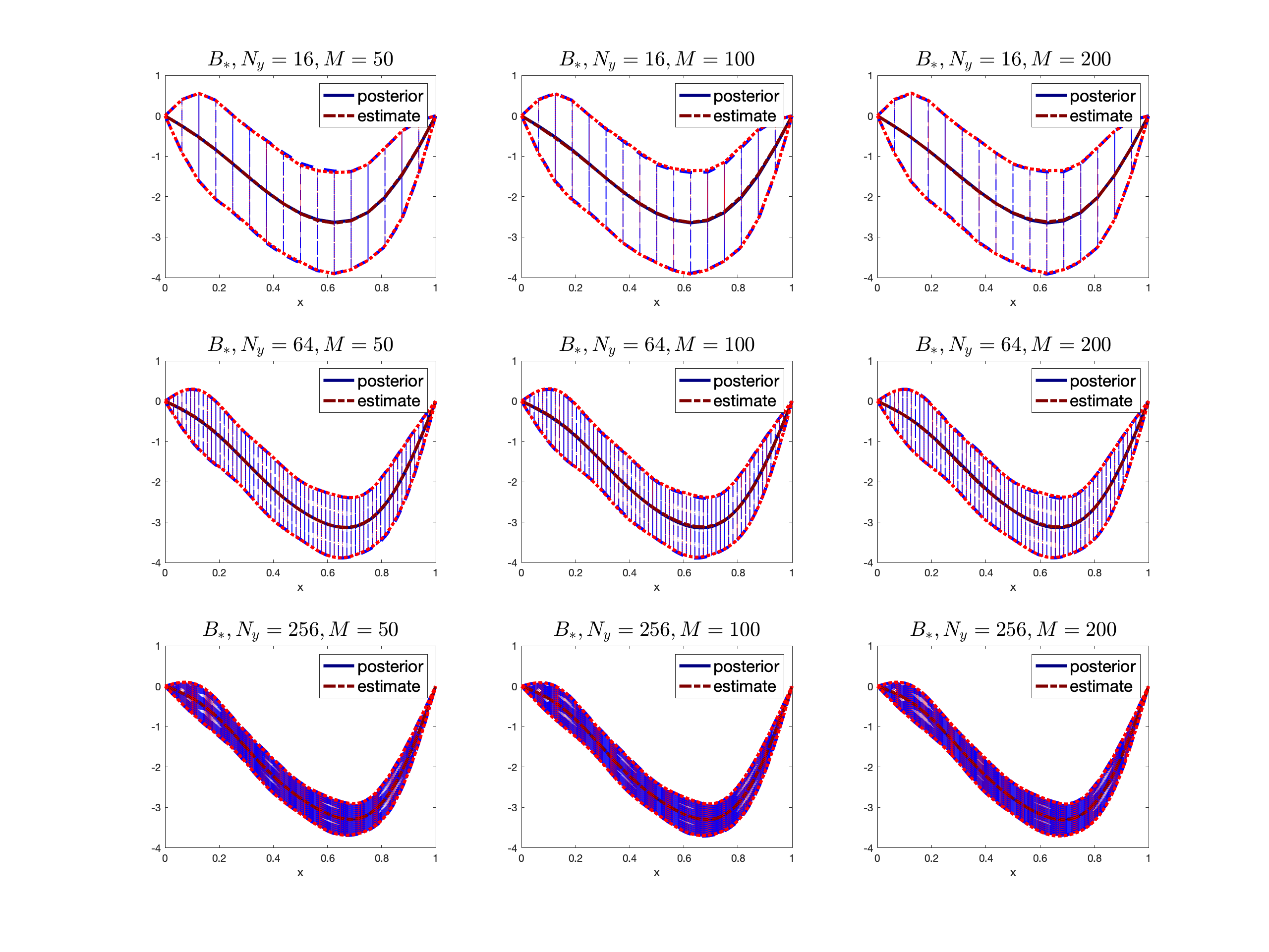}
    	\caption{{Approximations to the unknown parameter from the deterministic Fokker--Planck dynamics with $B_\ast$ for different choices of $N_y\in\{16,64,256\}$, $M\in\{50,100,200\}$ and fixed $N_x=4$.}}\label{fig:scaling_l_B2}
\end{figure}

\begin{figure}[!htb]
	\includegraphics[width=1\textwidth]{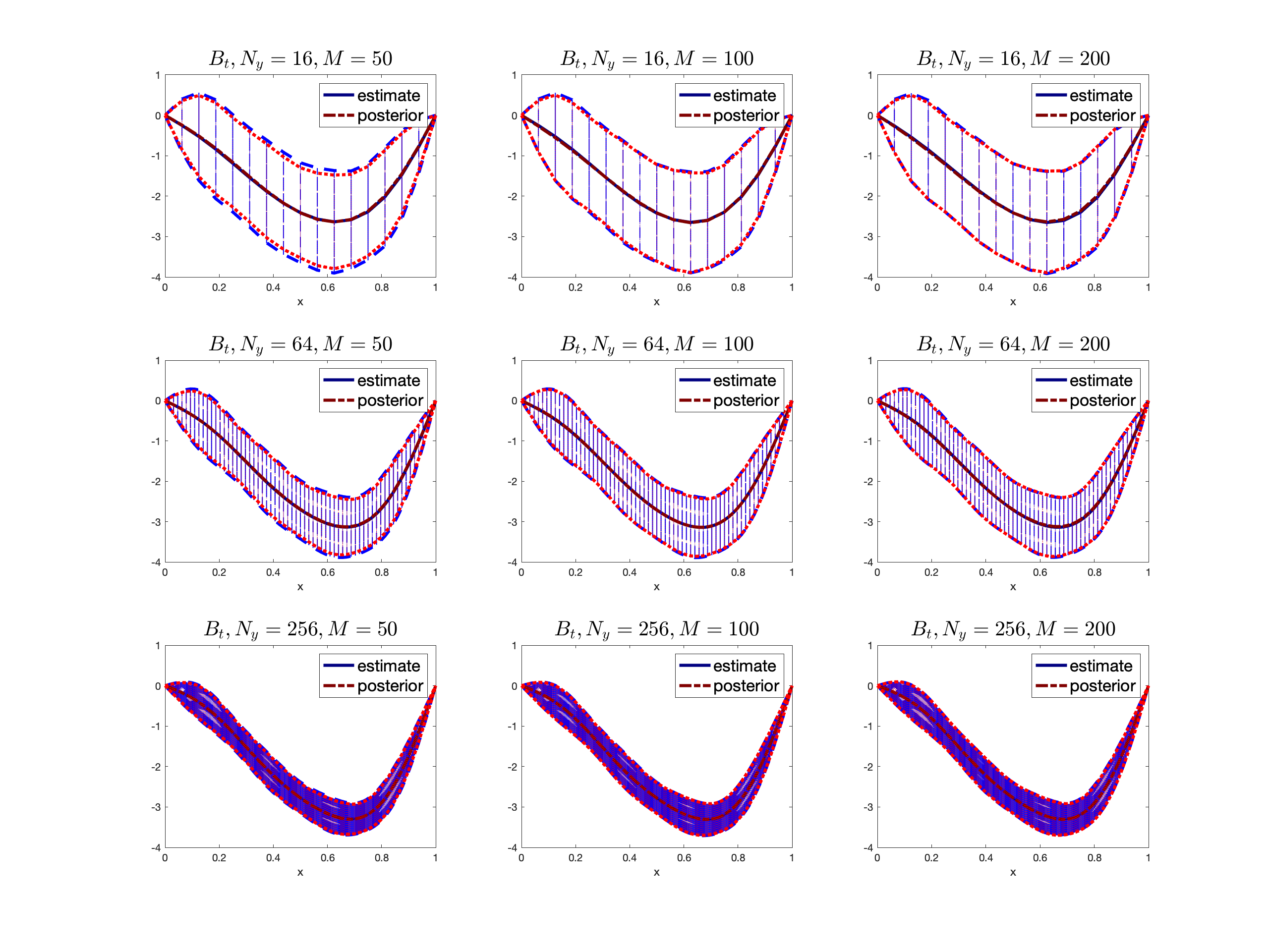}
    	\caption{{Approximations to the unknown parameter from the deterministic Fokker--Planck dynamics with $B_t$ for different choices of $N_y\in\{16,64,256\}$, $M\in\{50,100,200\}$ and fixed $N_x=4$.}}\label{fig:scaling_l_B3}
\end{figure}

\begin{table}
\begin{center}
\begin{tabular}{|c|c| c c c|c|}
\hline
$N_y \setminus M$ & & $50$ & $100$ & $200$ & theoretical\\
\hline
 $16$ & & $0.754$ & $0.707$ & $0.668$ & $0.396$\\%
 $64$ & $B$ & $0.703$  & $0.653$ &  $0.606$& $0.151$\\%
 $256$ & & $0.728$ & $0.657$  & $0.605$ & $0.046$\\%
\hline
$16$ & &  $0.3873$ & $0.404$ & $0.399$ & $0.396$\\
 $64$ & $B_\ast$ & $0.151$ & $0.152$  & $0.151$ & $0.151$\\
 $256$ & & $0.046$ & $0.046$ & $0.046$ & $0.046$\\
\hline 
$16$ & & $0.333$   & $0.371$ & $0.386$ & $0.396$\\
$64$ & $B_t$ & $0.130$  & $0.142$ & $0.148$ & $0.151$\\
$256$ & & $0.041$  & $0.045$ & $0.045$ & $0.046$\\
\hline
 \end{tabular}
\bigskip
\caption{Trace of the estimated covariance in comparison for the three different choices of $\{B,B_\ast,B_t\}$ and $N_y\in\{16,64,256\}$ and fixed $N_x=4$.}
\label{table:trace_l}
\end{center}
\end{table}

\begin{figure}[!htb]
	\includegraphics[width=1\textwidth]{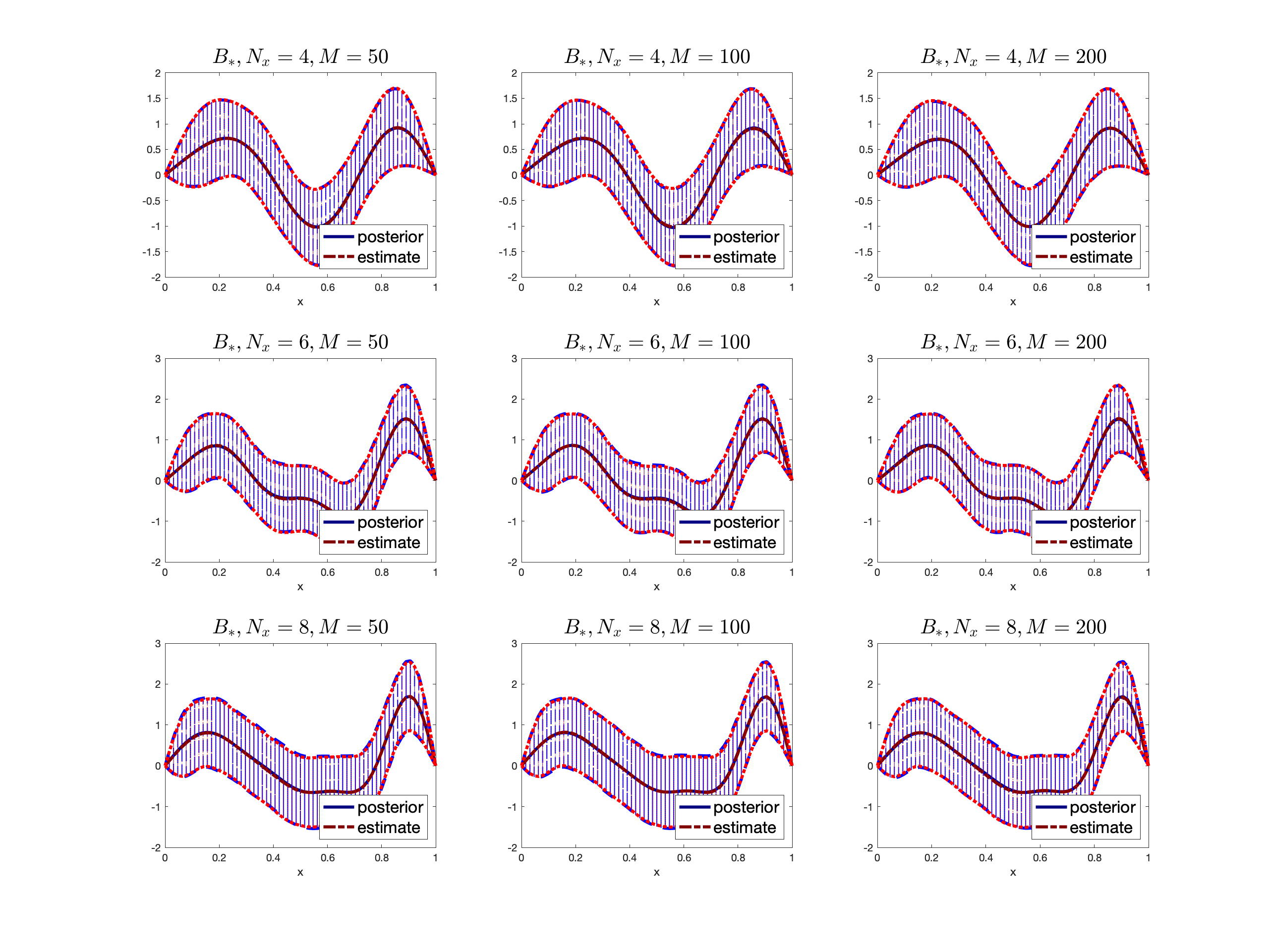}
    	\caption{{Approximations to the unknown parameter from the deterministic Fokker--Planck dynamics with $B_\ast$ for different choices of $N_x\in\{4,6,8\}$, $M\in\{50,100,200\}$ and fixed observations $N_y=64$.}}\label{fig:scaling_Nx_B2}
\end{figure}

\begin{figure}[!htb]
	\includegraphics[width=1\textwidth]{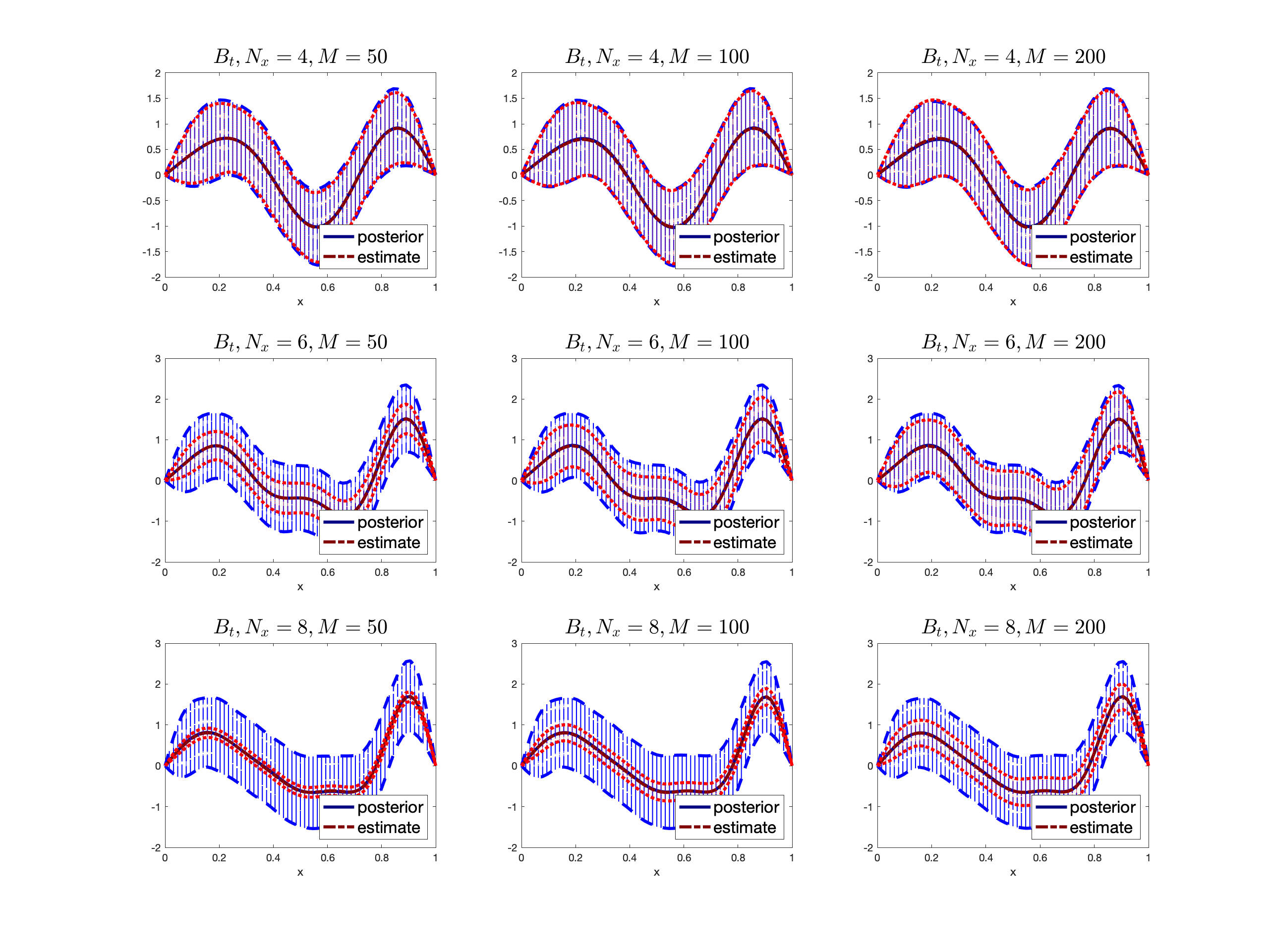}
    	\caption{{Approximations to the unknown parameter from the deterministic Fokker--Planck dynamics with $B_t$ for different choices of $N_x\in\{4,6,8\}$, $M\in\{50,100,200\}$ and fixed observations $N_y=64$.}}\label{fig:scaling_Nx_B3}
\end{figure}

\begin{table}
\begin{center}
\begin{tabular}{|c|c| c c c|c|}
\hline
$N_x \setminus M$ & & $50$ & $100$ & $200$ & theoretical\\
\hline
$4$ & &  $0.727$ & $0.650$  & $0.602$ & $0.151$\\
 $6$ & $B_1$ & $0.873$  &  $0.802$  &  $0.768$ & $0.191$\\
 $8$ & & $0.991$ &  $0.939$  &  $0.874$ & $0.217$\\
\hline
$4$ & &  $0.152$ & $0.152$ & $0.151$ & $0.151$\\
 $6$ & $B_2$ & $0.183$ & $0.186$ & $0.188$ & $0.191$\\
 $8$ & & $0.206$ & $0.207$  & $0.210$ & $0.217$\\
\hline 
$4$ & & $0.124$ &  $0.139$ & $0.147$ & $0.151$\\
$6$ & $B_3$ & $0.038$ & $0.080$ & $0.124$ & $0.191$\\
$8$ & & $0.004$ & $0.012$ & $0.030$ & $0.217$\\
\hline
 \end{tabular}
\bigskip
\caption{Trace of the estimated covariance in comparison for the three different choices of $\{B,B_\ast,B_t\}$, $N_x\in\{4,6,8\}$ and fixed $N_y=64$.}
\label{table:trace_Nx}
\end{center}
\end{table}

{
To illustrate our numerical results, we have created for both settings a table where we compare the trace of the estimated covariance for the posterior distribution. In Table~\ref{table:trace_l} we keep $N_x=4$ fixed and increase the number of observations $N_y$ from $16$ to $256$, whereas in Table~\ref{table:trace_Nx} we keep $N_y=64$ fixed and increase the dimension$N_x$ of the truncation from $4$ to $8$. The covariance has been estimated by solving \eqref{eq:dynamic_regFPf_nonlinear_2} up to a fixed time of $T=1000$ and producing samples according to \eqref{eq:approx_limit} using the final particle system. The ODE has been solved again with the MATLAB solver {\rm ode45}. For keeping the size $N_x$ of the parameter fixed, we observe, that we estimate the trace of the posterior covariance closely exact for $B_\ast$ and $B_t$ independently of the choices of $N_y\in\{16,64,256\}$, while $B$ overestimates for each choice. In contrast, we find that for fixed observations $N_y=64$ the choice of $B_t$ needs to include a larger ensemble size for increasing parameter dimension in order to estimate the trace of the posterior covariance correctly. The choice $B_\ast$ performs again very well, whereas the choice $B$ fails. These results can also be observed in Figure~\ref{fig:scaling_l_B2}-Figure~\ref{fig:scaling_l_B3} for the scaling in $N_y$ and in Figure~\ref{fig:scaling_Nx_B2}-Figure~\ref{fig:scaling_Nx_B3} for the scaling in $N_x$, where we compare the evaluation of the KL expansion \eqref{eq:KL_expansion} of the kernel based methods to the evaluation of the posterior distribution. We can see again that for the choice $B_t$, the sample size has to be increased for increaing $N_x$, while for fixed $N_x$ and increasing $N_y$ the obtained approximation results are stable.}

{
In summary, we observe that the deterministic Fokker--Planck particle system is a promising method as long as it is possible to choose RKHS representing the posterior distribution well. If there is no information available about the covariance structure of the posterior, it is a challenging task to choose a satisfying kernel. For example one could try to tune the choice $B_t$ by applying variance inflation or averaging over the last iterations in order to obtain a more accurate estimate for the variance through the particle system.
}


\subsection{High dimensional example}\label{subsec:infdim}

We return to the one-dimensional elliptic boundary-value problem from Section \ref{subsec:2d} now in the form of
\begin{equation}\label{eq:darcyflow}
-\frac{{\rm d}}{{\rm d}s}\left(\exp(u(s))\frac{{\rm d}}{{\rm d}s}p(s)\right) = 1, \quad s\in[0,1].
\end{equation}
While in the first example we had two unknown parameters arising from the unknown permeability constant $a = \exp(x_1)$ and the boundary condition $p(1) = x_2$, 
we consider here (\ref{eq:darcyflow}) for given boundary conditions $p(0)=p(1)=0$ and unknown permeability function $a(s) = \exp(u(s))$ with $u \in L^\infty([0,1])$. 
Note that for any $u\in L^\infty([0,1])$ there exists a solution  $p\in H_0^1([0,1];\R)$. Inspired by \cite{AGFHWLAS2019} we infer the coefficients 
of a Karhunen--Lo\`{e}ve (KL) expansion of $u$. More precisely, we assume that the prior of the unknown $u$ is given by the Gaussian process 
${\rm GP}(0,(-\Delta)^{-\tau})$. Thus, we can write $u\sim {\rm GP}(0,(-\Delta)^{-\tau})$ {again a.s.~through the KL expansion \eqref{eq:KL_expansion}.}

The inverse problem is to recover the coefficients $(x_k)_{k\in\N}$ corresponding to the KL expansion \eqref{eq:KL_expansion} given discrete noisy observations 
$y$ of \eqref{eq:darcyflow}, that is, $y = (\cO\circ G)(u(s,x))+\xi$. Here $G:\ L^\infty([0,1])\to H_0^1([0,1];\R)$ denotes the solution operator of \eqref{eq:darcyflow} and $\cO:\ H_0^1([0,1];\R)\to\R^{N_y}$ denotes the observation operator, which provides function values at $N_y$ equidistant observation points in $[0,1]$, that is, 
$z(\cdot)\in H_0^1([0,1];\R)\mapsto \cO(z(\cdot)) = (z(s_1),\dots,z(s_{N_y}))^{\rm T}$, $s_i = \frac{i}{N_y},\ i=1,\dots,N_y$. 

We will truncate the KL expansion \eqref{eq:KL_expansion} at index $N_x$, that is, the unknown parameters are 
given by $x = (x_1,\dots,x_{N_x})^{\rm T}\in\R^{N_x}$ and we define the forward map $h$ by
\begin{equation}
h:\ {\R}^{N_x}\to{\R}^{N_y}, \quad \text{with}\quad x\mapsto u(\cdot,x)\mapsto (\cO\circ G)(u(\cdot,x)).
\end{equation}

For our numerical results we replace $G$ by an numerical solution operator for \eqref{eq:darcyflow} on the grid $\cD\subset[0,1]$ with mesh size $h=2^{-8}$ 
and restrict $u(\cdot,x)$ to the computational grid $s_l = l\,h$, $l=1,\ldots,2^8-1$.

We set the prior to $X_0 \sim \cN(0,P_0)$, where $P_0\in\R^{N_x\times N_x}$ is diagonal matrix with entries $\lambda_k$, $k=1,\dots,N_x$ and $\tau = 1.5$, 
The measurement errors are mean zero Gaussian, that is, $\Xi \sim \cN(0,R)$ with $R= 0.01\cdot I_{N_y}$. The resulting inverse problem is of the form
(\ref{eq:IP}). We construct our reference data $y$ by drawing $x_k^\dagger\sim\cN(0,\lambda_k)$ for $k\le4$ and set ${x_k^\dagger}=0$ for $k>4$. We finally 
draw a measurement error $\xi^\dagger\sim\cN(0,R)$ and compute $y=h(x^\dagger)+\xi^\dagger$. In our numerical experiments, 
we truncate the KL expansion \eqref{eq:KL_expansion} at $N_x = 32$ and observe the solution $p(s)$ of \eqref{eq:darcyflow} at $N_y = 16$ equidistant grid points.

We again and use the MATLAB solver {\rm ode45} to time-step the deterministic Fokker--Planck dynamics \eqref{eq:dynamic_regFPf_nonlinear_2} 
and the Euler--Maruyama scheme {this time with an adaptive step-size $\Delta t_k \le 0.1/\beta_k$} for the interacting Langevin sampler \eqref{eq:interacting_langevin_corrected_localised}. {Here, $\beta_k$ is chosen such that $$ \beta_k = \max\{\|P_{t_k}^{xh} R^{-1}(h(X_{t_k}^{(i)})-y)+P_{t_k}^{xx}P_0^{-1}(X_{t_k}^{(i)}-\bar x_0)\|,\quad i = 1,\dots, M\}. $$}

This time we use a Gaussian kernel (\ref{eq:Gauss_kernel}) depending on the current empirical covariance matrix, {that is,  
$B_t = c_\delta/M^{\delta} \diag((P_t^{xx})_{ii})$, with $\delta$ and $c_\delta$ defined in \eqref{eq:bandwdth}, for the preconditioned Fokker--Planck dynamics. To reduce the issue of underestimating the variance through the choice $B_t$, we have fixed adapting the kernel $B_t$ at time $t=1$. This is, we choose $B=B_t$ for $t\le1$ and $B=B_1$ for $t\ge1$.}

Since our setup is quite similar to the high dimensional example in \cite{AGFHWLAS2019}, we will report on the numerical results using 
similar summary statistics.

We run both the deterministic Fokker--Planck dynamics \eqref{eq:dynamic_regFPf_nonlinear_2} as well as the interacting Langevin dynamic \eqref{eq:interacting_langevin_corrected_localised} {up to a fixed time of $T = 100$ with the initial ensemble  drawn i.i.d.~from the prior distribution $N(0,P_0)$.} 
To analyse the numerical results, we construct an empirical approximation to the posterior $X\mid y$ by a sample of size 
$L = 512$ for both methods. For the deterministic Fokker--Planck dynamics, we use the particle system at final time and produced samples according to \eqref{eq:approx_limit}. 
In the case of the Langevin sampler, we use the temporal evolution paths of the particles to collect the required total number samples. 
This has been achieved by adding the current ensemble of particles every {$10$} time steps to the already collected samples once the dynamics can be considered as
equilibrated. We compare our results with the posterior approximation resulting from a Random Walk Metropolis Hastings algorithm with preconditioned Crank--Nicolson 
(pCN) proposal \cite{sw:C13}, that is, we propose for given state $X_k$
\begin{equation}
\hat X_{k+1}\sim \cN(\sqrt{1-s^2}X_k,s^2P_0),
\end{equation}
where $s$ is a step size parameter. We set the step size to $s=0.07$, which results in an acceptance rate of approximately $25\%$ \cite{sw:RGG97}.

{As additional experiment we test a sequential Monte Carlo method (SMC) \cite{sw:DDJ2006, sw:KBJ2014}, which we combine with the interacting Langevin dynamic \eqref{eq:interacting_langevin_corrected_gradientfree}.  In particular, we initialize by $M$ particles $(X_0^{(i)})$ drawn from $\pi_0$, compute weights $W_0^{(i)}=1/M,\ i=1,\dots,M$ and proceed as follows for $n=1,\dots,N$:
\begin{itemize}
\item Importance weights: update weights by $W_n^{(i)} \propto W_{n-1}^{(i)} \cdot \pi_n$, with $\sum\limits_{i=1}^M W_n^{(i)} = 1$ and $\pi_n(x) = \frac{n}{N}\|y-h(x)\|_{R}^2$.
\item If the effective sample size $(\sum\limits_{i=1}^M (W_n^{(i)})^2)^{-1}<M_{\rm{tol}}$, we resample according to the weights $(W_n^{(i)})$.
\item we update the particles $X_{n-1}^{(i)} \mapsto X_n^{(i)}$ by running the interacting Langevin dynamic with scaled drift into direction of the data initialized by $X_{n-1}^{(i)}$ and up to time $t_n$:
\begin{subequations}\label{eq:interacting_langevin_corrected_gradientfree_SMC}
\begin{align}
{\rm d}X_t^{(i)} &= -\left\{\frac{n}{N}P_t^{xh}R^{-1}(h(X_t^{(i)})-y) + P_t^{xx}P_0^{-1}(X_t^{(i)}-\overline{x}_0) \right\}{\rm d}t\\
&\qquad \qquad 
+\,\nabla_{x^{(i)}} \cdot P_t^{xx}\,{\rm d}t+\sqrt{2}(P_t^{xx})^{1/2}\,{\rm d}W_t^{(i)}.
\end{align}
\end{subequations}
\end{itemize}
We compare the SMC with the interacting Langevin dynamic \eqref{eq:interacting_langevin_corrected_gradientfree} by running both methods up to a final time $T=0.5$ with $M=512$ particles. The SMC method will be simulated for $N=250$, with $t_n = 0.002$, $n=1,\dots,N$ and $M_{\rm tol} = 256$, such that it runs up to final time $T=0.5$, too.
}

\begin{figure}[!htb]
	\includegraphics[width=1\textwidth]{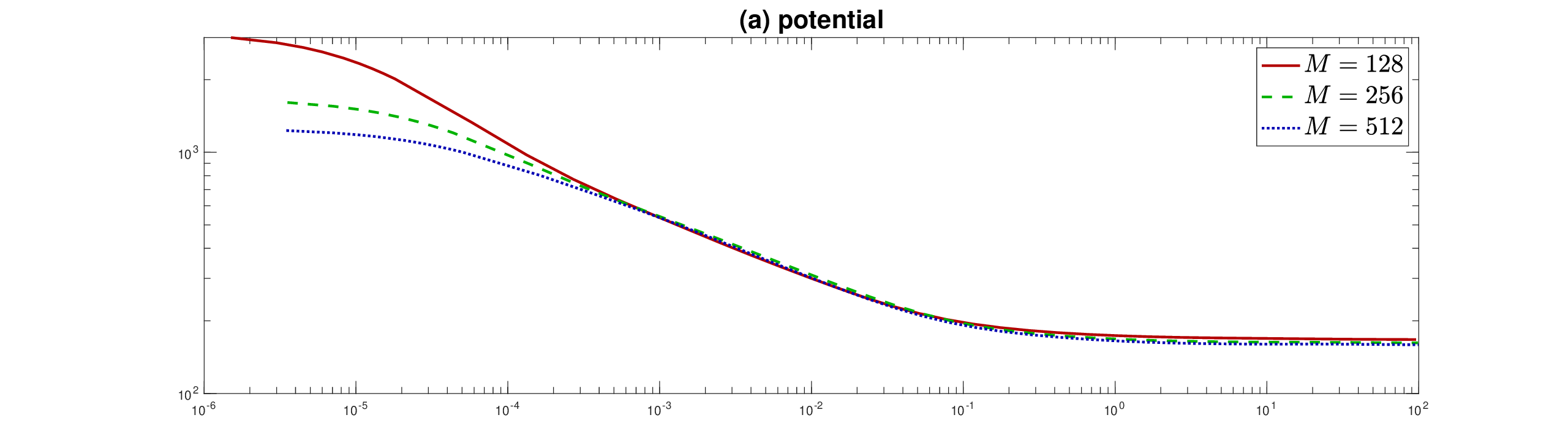}
    	\caption{{Time evolution of the potential function $\cV$ for different choices of $M$ (a).}}\label{fig:infdim_potential_different_alpha}
\end{figure}

{Figure \ref{fig:infdim_potential_different_alpha} shows the time evolution of the potential function \eqref{eq:potential_FP} for different choices of ensemble sizes $M$. There is no crucial sensitivity on those parameters detectable as all parameter choices 
result in quite a similar behaviour. In contrast to the temporal behaviour of the potential, we can see significant differences in the particle distributions by viewing scatter plots. }
{Figure \ref{fig:infdim_particle_system_different_M}  shows samples resulting from different choices of ensemble size $M$, which shows that a large ensemble size is necessary to produce good approximations to the posterior distribution.
} 


In the case of the Langevin sampler the effect of the correction term \eqref{eq:correction_term} introduced in \cite{sr:NR19} in the case of small ensemble sizes $M$
can be clearly detected. While the resulting samples without correction term concentrates in a small area, see Figure \ref{fig:infdim_EKS_sample_different_M}, 
the resulting sample with corrected dynamics fit the target distribution quite well for an ensemble size $M=16$ already; 
see Figure \ref{fig:infdim_EKS_corrected_sample_different_M}. The difference can also be seen in the time evolution of the spread of the ensemble over time
\begin{equation}
e_t := \frac1M\sum\limits_{i=1}^M |X_t^{(i)}-\bar X_t|^2,
\end{equation}
as shown in Figure \ref{fig:infdim_EKS_spread}. 
{The particles of the ensemble stay spread for different choices of ensemble sizes under the corrected dynamics, while we can see again the concentration effect for the uncorrected case. Further, we compare the resulting parameter 
estimation for both the deterministic as well as the stochastic method. The plot on the left in Figure \ref{fig:infdim_parameter_est} shows the resulting estimate
 for the deterministic version, that is, the evaluation of the truncated KL expansion \eqref{eq:KL_expansion}
\begin{equation}
u^{(i)} = u(\cdot, X^{(i)}) = \sum\limits_{k=1}^{N_x} X_k^{(i)}\psi_k(\cdot).
\end{equation} 
Similarly, we can see the estimates resulting from the stochastic method in the plot on the right. While the stochastic method fits the mean corresponding to the Random Walk Metropolis Hastings algorithms fairly well, the deterministic method seems to need a higher ensemble size to find the posterior distribution.}


{We close the discussion by analyzing the effect of the incorporation of  the SMC method in the interacting Langevin dynamics. We find that the inclusion of the resampling can accelerate the convergence to the posterior distribution. This can firstly be seen in the spread over time, Figure~\ref{fig:infdim_SMC_spread}, where we see that the occuring resampling is shifting the particle system into direction of the posterior distribution. In Figure~\ref{fig:infdim_particle_system_SMC} we see that already after final time $T=0.5$ the SMC method produces very good approximations of the posterior distribution, while the interacting Langevin dynamics still spreads too much. This result can also be seen in the parameter estimation in Figure~\ref{fig:infdim_parameter_est_SMC}.}

To summarise the numerical experiment, we have seen that in the case of the deterministic Fokker--Planck dynamics the choice of the kernel is crucial. While in the low dimensional examples it was enough to tune the parameter $\alpha$ for the kernel covariance $B = \alpha P_0^{xx}$, we had to introduce a time--dependent 
kernel $B_t = \alpha P_t^{xx}$ in the high dimensional example.  We have demonstrated the improvement of the interacting Langevin dynamics through 
the correction term \eqref{eq:correction_term}. While it is enough to chose a small ensemble size for the corrected sampler, we have to increase the ensemble size 
to $M \gg N_x$ in the method without correction.


\begin{figure}[!htb]
	\includegraphics[width=1\textwidth]{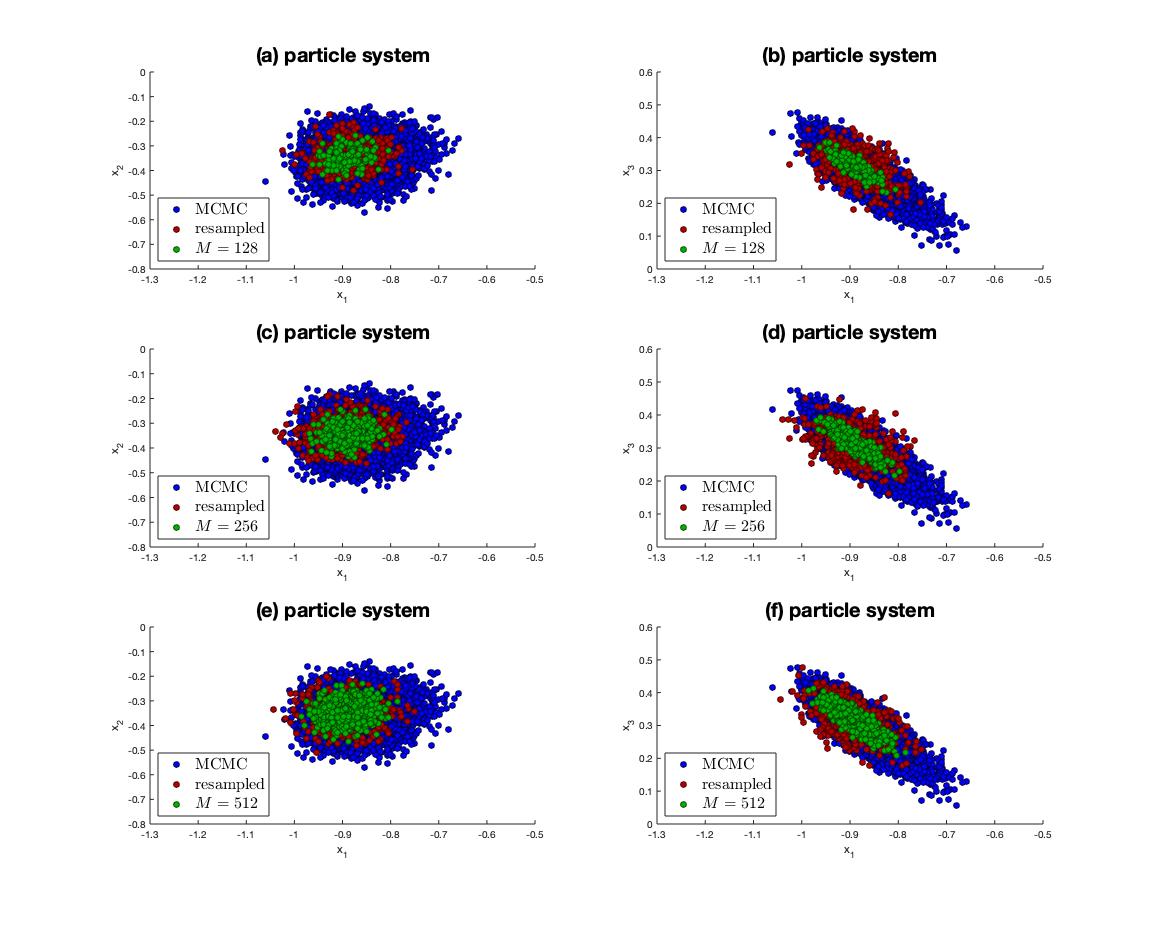}
    	\caption{{Approximations to two different marginals of the posterior PDF from deterministic Fokker--Planck dynamics for different values of
	the ensemble size $M$ and fixed kernel parameter $\alpha$: a) - b) $M = 128$, c) - d) $M = 256$, e) - f) $M = 512$.}}\label{fig:infdim_particle_system_different_M}
\end{figure}

\begin{figure}[!htb]
	\includegraphics[width=1\textwidth]{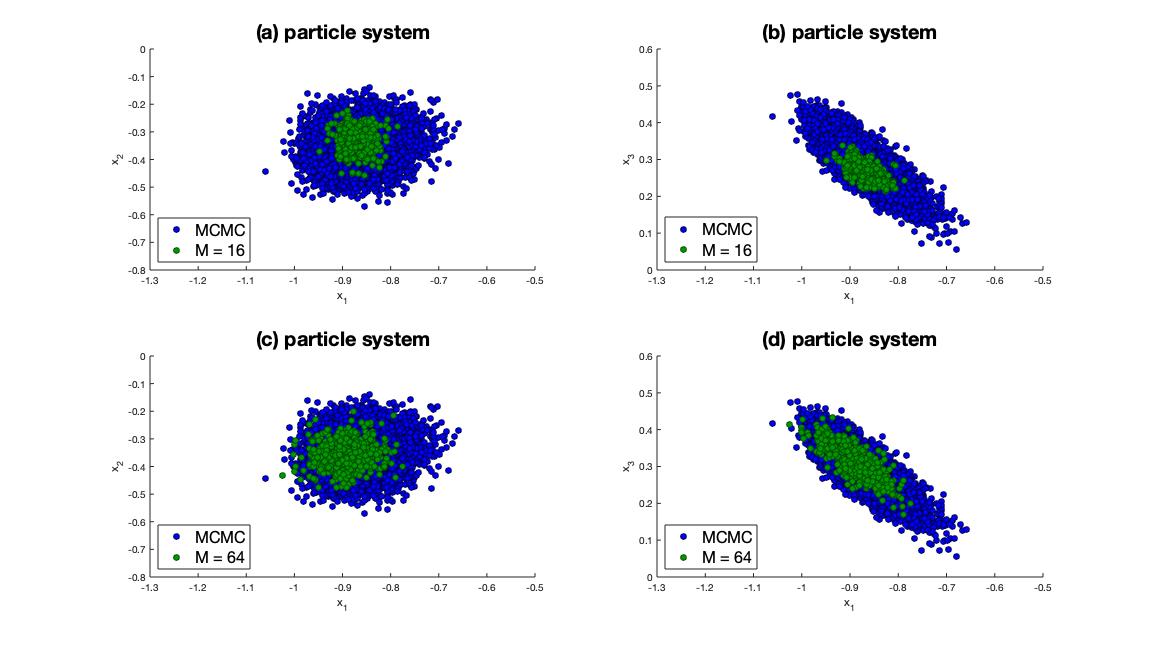}
    	\caption{{Approximations to two different marginals of the posterior PDF from interacting Langevin dynamics without correction for different values of
	the ensemble size $M$: a) - b) $M = 16$, c) - d) $M = 64$.}}\label{fig:infdim_EKS_sample_different_M}
\end{figure}
\begin{figure}[!htb]
	\includegraphics[width=1\textwidth]{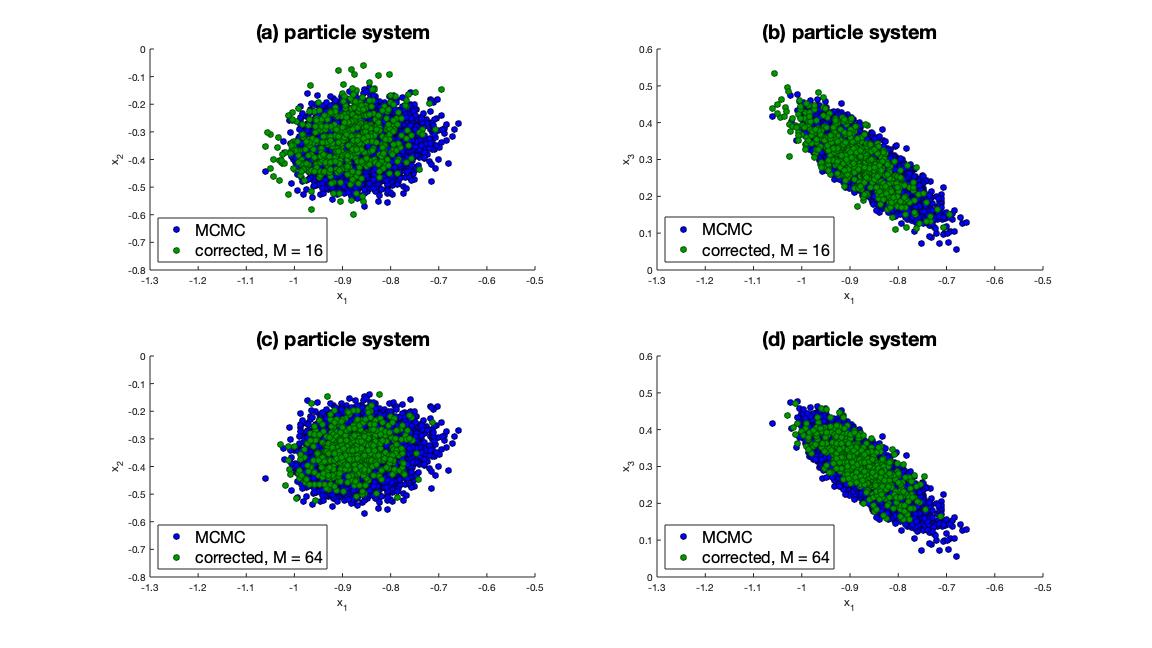}
    	\caption{{Approximations to two different marginals of the posterior PDF from interacting Langevin dynamics with correction for different values of
	the ensemble size $M$: a) - b) $M = 16$, c) - d) $M = 64$.}}\label{fig:infdim_EKS_corrected_sample_different_M}
\end{figure}

\begin{figure}[!htb]
	\includegraphics[width=1\textwidth]{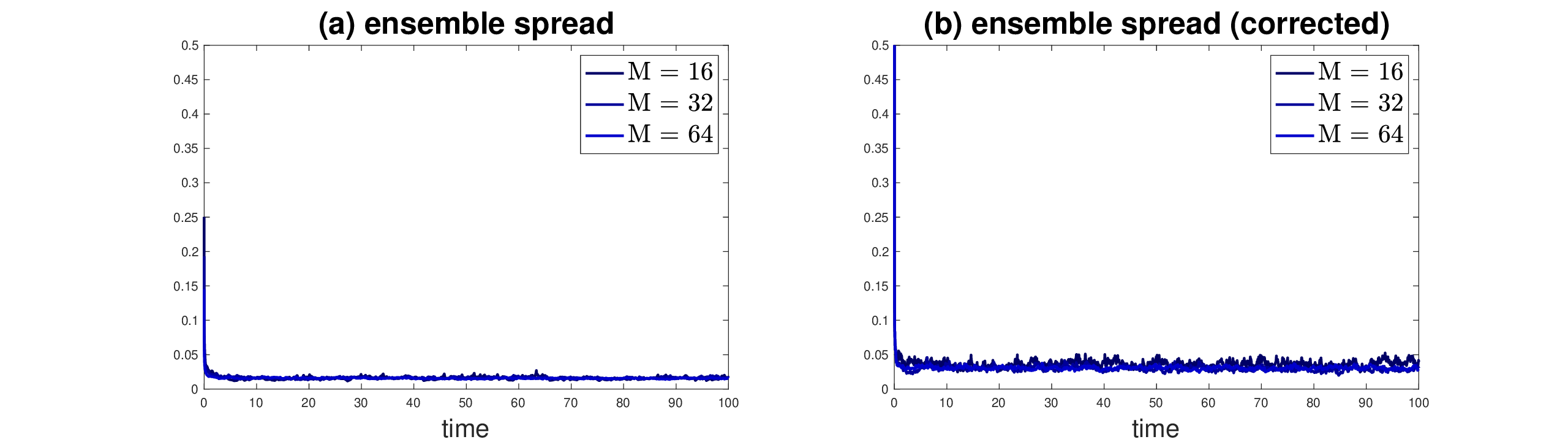}
    	\caption{Comparison of the spread of the ensemble from interacting Langevin dynamics with and without correction over time: a) without correction, b) with correction.}\label{fig:infdim_EKS_spread}
\end{figure}

\begin{figure}[!htb]
	\includegraphics[width=1\textwidth]{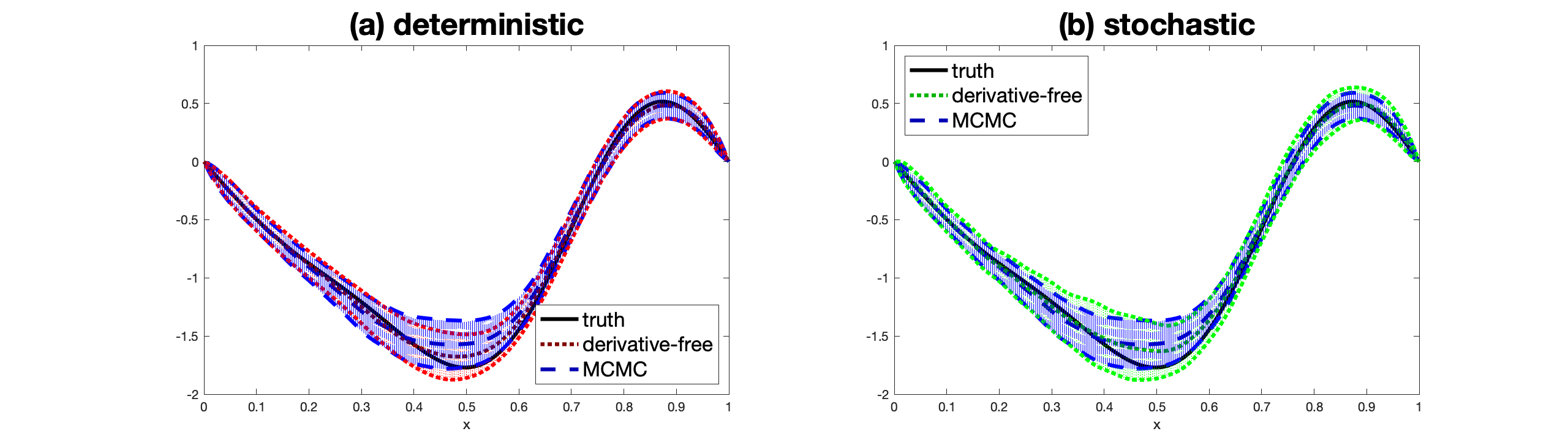}
    	\caption{{Approximations to the unknown parameter from: a) deterministic Fokker--Planck dynamics, b) interacting Langevin dynamics.}}\label{fig:infdim_parameter_est}
\end{figure}

\begin{figure}[!htb]
	\centering
	\includegraphics[width=0.5\textwidth]{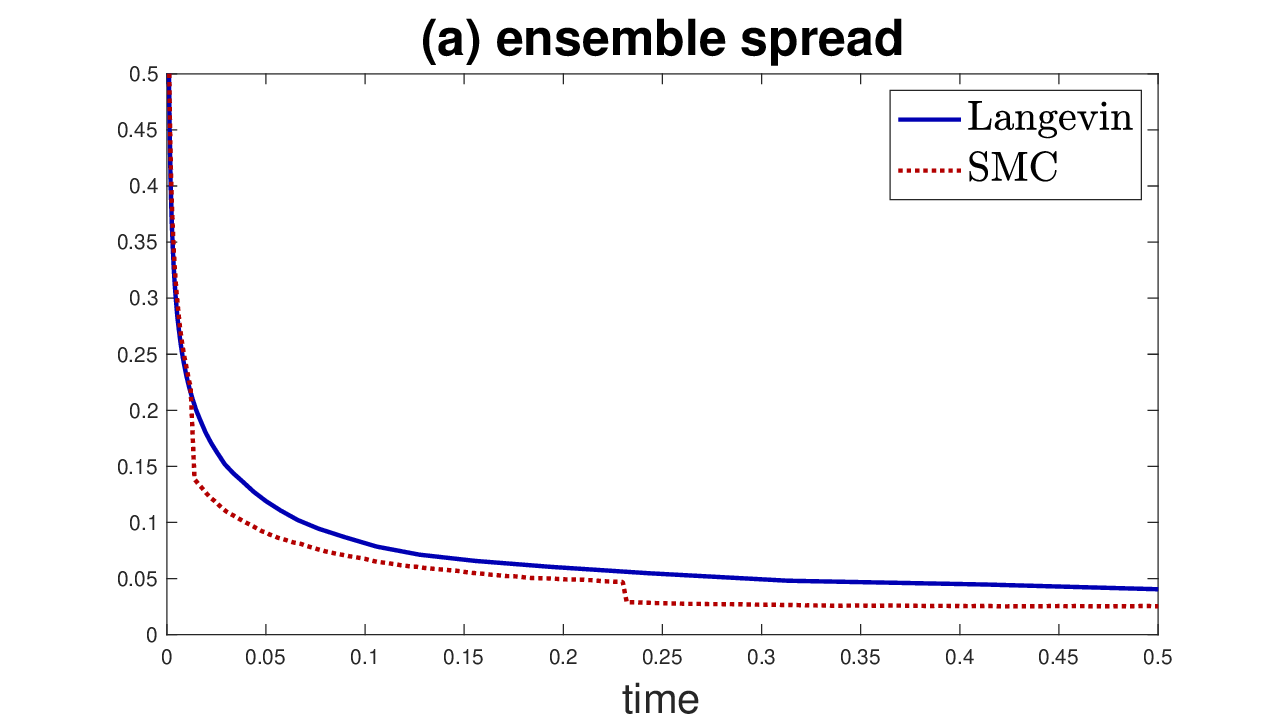}
	
    	\caption{{Comparison of the spread of the ensemble from the sequential Monte Carlo method and the interacting Langevin dynamics.}}\label{fig:infdim_SMC_spread}
\end{figure}

\begin{figure}[!htb]
	\includegraphics[width=1\textwidth]{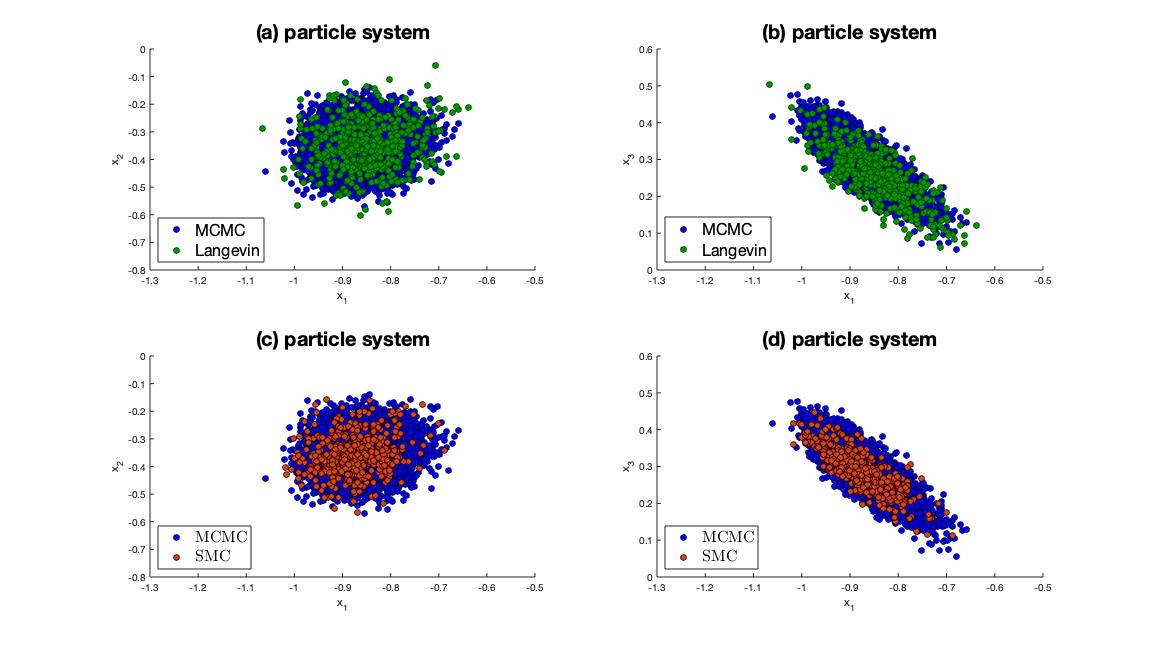}
    	\caption{{Approximations to two different marginals of the posterior PDF from sequential Monte Carlo method fixed ensemble size $M=512$: a)-b) interacting Langevin dynamics c)-d) sequential Monte Carlo method}}\label{fig:infdim_particle_system_SMC}
\end{figure}

\begin{figure}[!htb]
	\includegraphics[width=1\textwidth]{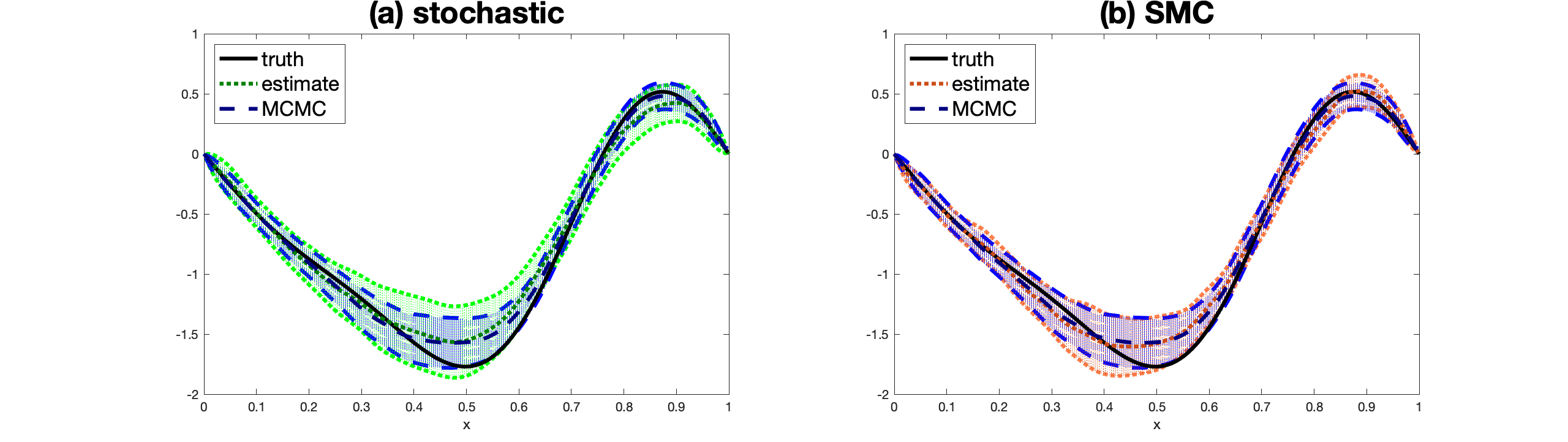}
    	\caption{{Approximations to the unknown parameter from: a) interacting Langevin dynamics, b) sequential Monte Carlo method.}}\label{fig:infdim_parameter_est_SMC}
\end{figure}


\section{Conclusions} \label{sec:conclusions}

We have discussed deterministic and stochastic interacting particle approximations to the Fokker--Planck formulation of 
overdamped Langevin dynamics (Brownian dynamics) and its application to BIPs. Preconditioning of such approximations by 
the empirical covariance matrix leads to affine invariant formulations and to the possibility of gradient-free implementations. 
We have proposed a localised preconditioning approach which allows for an application of such gradient-free formulations to BIPs with
multimodal posterior distributions. While the deterministic formulations depend crucially on the choice of the RKHS, the
stochastic formulations do not require such parametrisations; their convergence to equilibrium is however more difficult to track. {At the same time, the computational complexity of the underlying kernel-based potential and its gradient exceeds those arising from the stochastic formulations. Furthermore, kernel-based density approximations, which work well in low dimensions, are 
subject to the curse of dimensionality.}

{Further work is required on efficient time-stepping methods for the proposed interacting particle systems. Despite being
affine invariant, the equations of motion can be stiff, that is, it may require very small times-steps with an explicit time-stepping method, 
when initialised from a distribution that is far from the desired target distribution as it is often the case in Bayesian inference.} 

We finally mention that both the deterministic and stochastic interacting particle formulations can be combined with
stochastic gradient descent (SGD) \cite{sr:BCN18} and stochastic gradient Langevin dynamics (SGLD) \cite{sr:WT11}, respectively, methods. This aspect will be explored further in a forthcoming publication.


\medskip

\noindent
{\bf Acknowledgement.} This work has been partially funded by Deutsche Forschungsgemeinschaft (DFG, German Science Foundation) - SFB 1294/1 - 318763901. SW is grateful to the DFG RTG1953 "Statistical Modeling of Complex Systems and Processes" for funding of this research. The authors acknowledge the support by the state of Baden--W\"urttemberg through bwHPC. We like to thank Sahani Pathiraja and Claudia Schillings for feedback on earlier drafts of the paper.


\bibliographystyle{siam}
\bibliography{mybib}

\begin{thebibliography}{10}

\bibitem{sr:BR10}
{\sc K.~Bergemann and S.~Reich}, {\em A localization technique for ensemble
  {K}alman filters}, Q. J. R. Meteorological Soc., 136 (2010), pp.~701--707.

\bibitem{sr:br11}
{\sc K.~Bergemann and S.~Reich}, {\em An ensemble {K}alman--{B}ucy filter for
  continuous data assimilation}, Meteorolog.~Zeitschrift, 21 (2012),
  pp.~213--219.

\bibitem{DBCSPWSW2019}
{\sc D.~Bl\"omker, C.~Schillings, P.~Wacker, and S.~Weissmann}, {\em Well
  posedness and convergence analysis of the ensemble {K}alman inversion},
  Inverse Problems,  (2019).

\bibitem{sr:BCN18}
{\sc L.~Bottou, F.~Curtis, and J.~Nocedal}, {\em Optimization methods for
  large-scale machine learning}, SIAM Review, 60 (2018), pp.~223--311.

\bibitem{sr:CCP19}
{\sc J.~Carrillo, K.~Craig, and F.~Patacchini}, {\em A blob method for
  diffusion}, Calculus of Variations and Partial Differential Equations, 58
  (2019), pp.~1--53.

\bibitem{sw:C13}
{\sc S.~L. Cotter, G.~O. Roberts, A.~M. Stuart, and D.~White}, {\em {MCMC}
  methods for functions: Modifying old algorithms to make them faster},
  Statist. Sci., 28 (2013), pp.~424--446.

\bibitem{sr:daum11}
{\sc F.~Daum and J.~Huang}, {\em Particle filter for nonlinear filters}, in
  Acoustics, Speech and Signal Processing (ICASSP), 2011 IEEE International
  Conference on, 2011, pp.~5920--5923.

\bibitem{sr:DM90}
{\sc P.~Degond and F.-J. Mustieles}, {\em A deterministic approximation of
  diffusion equations using particles}, SIAM J. Sci. Comput., 11 (1990),
  pp.~293--310.

\bibitem{sw:DDJ2006}
{\sc P.~Del~Moral, A.~Doucet, and A.~Jasra}, {\em Sequential {M}onte {C}arlo
  samplers}, Journal of the Royal Statistical Society: Series B (Statistical
  Methodology), 68 (2006), pp.~411--436.

\bibitem{DNS19}
{\sc A.~Duncan, N.~N\"usken, and L.~Szpruch}, {\em On the geometry of {S}tein
  variational descent}, Tech. Rep. arXiv:1912.00894, Imperial College London
  and University of Potsdam, 2019.

\bibitem{ErnstEtAl2015}
{\sc O.~G. Ernst, B.~Sprungk, and H.-J. Starkloff}, {\em Analysis of the
  ensemble and polynomial chaos {K}alman filters in {B}ayesian inverse
  problems}, SIAM/ASA Journal on Uncertainty Quantification, 3 (2015),
  pp.~823--851.

\bibitem{Evensen2003}
{\sc G.~Evensen}, {\em The ensemble {K}alman filter: theoretical formulation
  and practical implementation}, Ocean Dynamics, 53 (2003), pp.~343--367.

\bibitem{sr:evensen}
{\sc G.~Evensen}, {\em Data Assimilation. {T}he Ensemble Kalman Filter},
  Springer-Verlag, New York, 2006.

\bibitem{AGFHWLAS2019}
{\sc A.~Garbuno-Inigo, F.~Hoffmann, W.~Li, and A.~M. Stuart}, {\em Interacting
  {L}angevin diffusions: Gradient structure and ensemble {K}alman sampler},
  SIAM Journal on Applied Dynamical Systems, 19 (2020), pp.~412--441.

\bibitem{sr:INR19}
{\sc A.~Garbuno-Inigo, N.~N\"usken, and S.~Reich}, {\em Affine invariant
  interacting {L}angevin dynamics for {B}ayesian inference}, SIAM Journal on
  Applied Dynamical Systems, 19 (2020), pp.~1633--1658.

\bibitem{sr:GC11}
{\sc M.~Girolami and B.~Calderhead}, {\em Riemann manifold {L}angevin and
  {H}amiltonian {M}onte {C}alro methods}, J. R. Statist. Soc. B, 73 (2011),
  pp.~123--214.

\bibitem{sr:GW10}
{\sc J.~Goodman and J.~Weare}, {\em Ensemble samplers with affine invariance},
  Comm. Appl. Math. and Comput. Science, 5 (2010), pp.~65--80.

\bibitem{sr:H18}
{\sc J.~Harlim}, {\em Data--driven computational methods}, Cambridge University
  Press, Cambridge, 2018.

\bibitem{MHGV2018}
{\sc M.~Herty and G.~Visconti}, {\em Kinetic methods for inverse problems},
  Kinetic \& Related Models, 12 (2019), p.~1109.

\bibitem{StLawIg2013}
{\sc M.~A. Iglesias, K.~Law, and A.~M. Stuart}, {\em Ensemble {K}alman methods
  for inverse problems}, Inverse Problems, 29 (2013), p.~045001.

\bibitem{sr:JKO98}
{\sc R.~Jordan, D.~Kinderlehrer, and F.~Otto}, {\em The variational formulation
  of the {F}okker--{P}lanck equation}, SIAM Journal on Mathematical Analysis,
  29 (1998), pp.~1--17.

\bibitem{Kaipio:1338003}
{\sc J.~Kaipio and E.~Somersalo}, {\em {Statistical and Computational Inverse
  Problems}}, Springer, Dordrecht, 2005.

\bibitem{sw:KBJ2014}
{\sc N.~Kantas, A.~Beskos, and A.~Jasra}, {\em Sequential {M}onte {C}arlo
  methods for high-dimensional inverse problems: A case study for the
  {N}avier--{S}tokes equations}, SIAM/ASA Journal on Uncertainty
  Quantification, 2 (2014), pp.~464--489.

\bibitem{NKAS2018}
{\sc N.~B. Kovachki and A.~M. Stuart}, {\em Ensemble {K}alman inversion: a
  derivative-free technique for machine learning tasks}, Inverse Problems, 35
  (2019), p.~095005.

\bibitem{sr:stuart15}
{\sc K.~Law, A.~Stuart, and K.~Zygalakis}, {\em Data assimilation: {A}
  mathematical introduction}, Springer-Verlag, New York, 2015.

\bibitem{sr:LMW18}
{\sc B.~Leimkuhler, C.~Matthews, and J.~Weare}, {\em Ensemble preconditioning
  for {M}arkov chain {M}onte {C}arlo simulations}, Stat.~Comput., 28 (2018),
  pp.~277--290.

\bibitem{sr:LW16}
{\sc Q.~Liu and D.~Wang}, {\em Stein variational gradient descent: {A} general
  purpose {B}ayesian inference algorithm}, in Advances in Neural Information
  Processing Systems 29 (NIPS 2016), 2016, pp.~2378--2386.

\bibitem{sr:LLN18}
{\sc J.~Lu, Y.~Lu, and J.~Nolen}, {\em Scaling limit of the {S}tein variational
  gradient descent: {T}he mean field regime}, SIAM Math. Anal., 51 (2019),
  pp.~648--671.

\bibitem{sw:M16}
{\sc Y.~Marzouk, T.~Moselhy, M.~Parno, and A.~Spantini}, {\em Sampling via
  measure transport: An introduction}, in Handbook of Uncertainty
  Quantification, Springer International Publishing, New York, 2016, pp.~1--41.

\bibitem{sr:NR19}
{\sc N.~N\"usken and S.~Reich}, {\em Note on interacting {L}angevin diffusion:
  {G}radient structure and ensemble {K}alman sampler}, Tech. Rep.
  arXiv:1908.10890v1, University of Potsdam, 2019.

\bibitem{SPSR2019}
{\sc S.~{Pathiraja} and S.~{Reich}}, {\em {Discrete gradients for computational
  {B}ayesian inference}}, Journal of Computational Dynamics, 6 (2019),
  pp.~236--251.

\bibitem{sr:P14}
{\sc G.~Pavliotis}, {\em Stochastic processes and applications},
  Springer--Verlag, New York, 2014.

\bibitem{sr:reich10}
{\sc S.~Reich}, {\em A dynamical systems framework for intermittent data
  assimilation}, BIT Numer Math, 51 (2011), pp.~235--249.

\bibitem{SR2018}
{\sc S.~Reich}, {\em Data assimilation: The {S}chr\"odinger perspective}, Acta
  Numerica, 28 (2019), pp.~635--711.

\bibitem{sr:reichcotter15}
{\sc S.~Reich and C.~Cotter}, {\em Probabilistic forecasting and {B}ayesian
  data assimilation}, Cambridge University Press, Cambridge, 2015.

\bibitem{sw:SR2021}
{\sc S.~Reich and S.~Weissmann}, {\em
  https://github.com/{S}{W}eissma/{F}okker-{P}lanck-particle-systems/},
  (2021).

\bibitem{sw:RGG97}
{\sc G.~O. Roberts, A.~Gelman, and W.~R. Gilks}, {\em Weak convergence and
  optimal scaling of random walk metropolis algorithms}, Ann. Appl. Probab., 7
  (1997), pp.~110--120.

\bibitem{sr:R90}
{\sc G.~Russo}, {\em Deterministic diffusion of particles}, Comm. Pure Appl.
  Math., 43 (1990), pp.~697--733.

\bibitem{SchSt2016}
{\sc C.~Schillings and A.~M. Stuart}, {\em Analysis of the ensemble {K}alman
  filter for inverse problems}, SIAM Journal on Numerical Analysis, 55 (2017),
  pp.~1264--1290.

\bibitem{sw:S92}
{\sc D.~Scott}, {\em Multivariate Density Estimation: Theory, Practice, and
  Visualization}, A Wiley-interscience publication, Wiley, 1992.

\bibitem{stuart_2010}
{\sc A.~M. Stuart}, {\em Inverse problems: A {B}ayesian perspective}, Acta
  Numerica, 19 (2010), p.~451–559.

\bibitem{sr:S15}
{\sc T.~Sullivan}, {\em Introduction to uncertainty quantification},
  Springer-Verlag, New York, 2015.

\bibitem{sr:TM19}
{\sc A.~Taghvaei and P.~Mehta}, {\em Accelerated flow for probability
  distributions}, in Proceedings of the 36th International Conference on
  Machine Learning, K.~Chaudhuri and R.~Salakhutdinov, eds., vol.~97 of
  Proceedings of Machine Learning Research, Long Beach, California, USA, 09--15
  Jun 2019, PMLR, pp.~6076--6085.

\bibitem{sr:TMM19}
{\sc A.~{Taghvaei} and P.~G. {Mehta}}, {\em Gain function approximation in the
  feedback particle filter}, in 2016 IEEE 55th Conference on Decision and
  Control (CDC), 2016, pp.~5446--5452.

\bibitem{sw:W06}
{\sc L.~Wasserman}, {\em All of Nonparametric Statistics}, Springer Texts in
  Statistics, Springer New York, 2006.

\bibitem{sr:WT11}
{\sc M.~Welling and Y.~Teh}, {\em Bayesian learning via stochastic gradient
  {L}angevin dynamics}, in Proceedings of the 28th International Conference on
  Machine Learning, ICML'11, Omnipress, 2011, pp.~681--688.

\end{thebibliography}


\end{document}